\documentclass[9pt]{elsarticle}

\biboptions{sort}

\usepackage[a4paper,width=160mm,top=25mm,bottom=30mm]{geometry}

\graphicspath{ {pictures/} }
\usepackage{subcaption}
\usepackage{mdframed}   
\usepackage{booktabs}
\usepackage{afterpage}
\usepackage{amsthm}
\usepackage{mathtools}
\usepackage{framed} 
\usepackage{capt-of}
\usepackage[utf8x]{inputenc} 
\newtheorem{theorem}{Theorem}

\newtheorem{proposition}{Proposition}
\usepackage{graphicx,import}

\theoremstyle{remark}
\newtheorem{remark}{Remark}

\usepackage{esvect} 

\usepackage{tikz}
\usepackage{pgfplots}
\usepgfplotslibrary{colorbrewer}
\usetikzlibrary{spy}

\usepackage{graphicx}
\usepackage{amssymb}
\usepackage{mathtools}
\usepackage{amsmath}
\usepackage{commath} 
\mathtoolsset{showonlyrefs} 
\usepackage{amssymb}
\usepackage{siunitx} 
\mdfdefinestyle{box}{%
    outerlinewidth=0.5pt,
    innertopmargin=0pt,
    innerbottommargin=0pt,
    innerrightmargin=0pt,
    innerleftmargin=0pt,
    topline=false,
    rightline=false,
    leftline=false,
    bottomline=false,
    backgroundcolor=white}

\newcounter{weakforms}

\newenvironment{weakform}[1]
{
\refstepcounter{weakforms}
\begin{mdframed}[style=box]
\textbf{(W\theweakforms)}
\itshape
}
{
    \end{mdframed}
    }

\pgfplotsset{
	compat=newest,
	tick label style={font=\small},
	label style={font=\small},
	legend style={nodes={scale=0.7, transform shape}}
	}

\usepackage[colorlinks = true,
            linkcolor = blue,
            urlcolor  = blue,
            citecolor = blue,
            anchorcolor = blue]{hyperref}

\newcommand{\weakref}[1]
{\hyperref[#1]{\textrm{W\ref*{#1}}}}

\journal{arXiv}

\begin{document}

\begin{frontmatter}



\title{Coupling non-conforming discretizations of PDEs by spectral approximation of the Lagrange multiplier space}
\author[epfl]{Simone Deparis}
\author[epfl]{Luca Pegolotti\corref{cor1}}
\ead{luca.pegolotti@epfl.ch}

\cortext[cor1]{Corresponding author}

\address[epfl]{Institute of Mathematics, \'{E}cole Polytechnique
F\'{e}d\'{e}rale de Lausanne, Station 8, EPFL, CH–1015 Lausanne, Switzerland}

\begin{abstract}
This work focuses on the development of a non-conforming domain decomposition method for the approximation of PDEs based on weakly imposed transmission conditions: the continuity of the global solution is enforced by a discrete number of Lagrange multipliers defined over the interfaces of adjacent subdomains. The method falls into the class of primal hybrid methods, which also include the well-known mortar method. Differently from the mortar method, we discretize the space of basis functions on the interface by spectral approximation independently of the discretization of the two adjacent domains; one of the possible choices is to approximate the interface variational space by Fourier basis functions.
As we show in the numerical simulations, our approach is well-suited for the solution of problems with non-conforming meshes or with finite element basis functions with different polynomial degrees in each subdomain. Another application of the method that still needs to be investigated is the coupling of solutions obtained from otherwise incompatible methods, such as the finite element method, the spectral element method or isogeometric analysis.
\end{abstract}


\end{frontmatter}

\section{Introduction}
In numerical analysis, domain decomposition methods are techniques for the splitting of Partial Differential Equations (PDEs) into smaller and coupled problems defined over subsets of the original domain. The splitting may be motivated by physical reasons, for instance when the subdomains are characterized by different governing equations (e.g. in fluid-structure-interaction problems \cite{deparis2006domain}) or by discretization needs, should it be required to employ specific methods (e.g. finite element method or spectral element method) or specific polynomial degrees in certain regions of the domain \cite{toselli2005domain}. Moreover, domain decomposition methods have become particularly important for the solution of large scale problems on multiprocessors or clusters, as they allow the mapping of the subproblems on separate cores \cite{israeli1993domain}.

Domain decomposition methods are typically based either on iterative or direct procedures \cite{becker2003finite}. In the first class of techniques the continuity on the interfaces of the solution, of its normal derivatives or combinations of the two are strongly imposed. Typically, these methods require solving the problems defined on the subdomains separately multiple times while imposing artificial boundary conditions based on the solutions at the previous iteration. The type of boundary conditions employed on each subdomain is a peculiarity of each algorithm, so that the literature on the topic commonly refers to the Dirichlet-Dirichlet algorithm, the Dirichlet-Neumann algorithm, and so on; see e.g \cite{toselli2005domain} for details. These strategies allow reducing the size of the linear systems to be solved and, most importantly, to compute the solution on each subdomain in parallel.

In this paper, we present an approach belonging to the class of direct procedures in which  the continuity conditions (often called \textit{transmission conditions}) are weakly imposed through the use of suitable Lagrange multipliers. Our method is applied to PDEs written in primal hybrid formulation, and for this reason it shares some of the features of the well-known mortar method \cite{bernardi1989new,bernardi2005basics}. This was originally proposed to solve PDEs by combining spectral elements and finite elements, or by combining finite element spaces with different polynomial degrees, in non overlapping portions of the domain \cite{quarteroni1999domain}. Since then, the mortar method has become the non-conforming method of choice in many areas of computational science and engineering, for example in contact mechanics \cite{puso2004mortar}, solid mechanics \cite{puso20043d}, fluid mechanics \cite{ehrl2014dual} and fluid-structure interaction problems \cite{kloppel2011fluid}; see also \cite{popp2014dual,hesch2014mortar,belgacem2003hp}. The implementation of the mortar method is not straightforward, as the algorithm is based on $L^2$-projections of the traces of functional spaces defined on a group subdomains -- the masters -- onto the interfaces of the adjacent ones -- the slaves. INTERNODES \cite{deparis2016internodes, forti2016parallel}, a recently developed method for the treatment of non-conforming meshes, overcomes this issue by treating the transmission conditions with the interpolation of basis functions of the master domains onto the interfaces of the slaves.

As in the mortar method, our approach is based on the idea that the global problem can be subdivided into a set of smaller problems coupled with weak conditions relying on basis functions defined on the interfaces. In the mortar method, such basis functions are obtained from the trace space of the adjacent slave domains. This choice is convenient from the analysis standpoint but makes the implementation of the method cumbersome. Another drawback is that the final solution is dependent on the choice of master and slave domains. The originality of our method is to consider basis functions on the interfaces which are completely independent of the discretization of the neighboring domains: in this paper, we employ spectral basis functions (specifically, Fourier basis functions).  This comes with the advantage of obtaining a solution which is indifferent to the choice of master and slave domains. Moreover, the accuracy of the coupling of solutions at the interfaces is easily tuned by varying the number of basis functions on the common boundary. Our approach can be interpreted as a specialization of the three-fields method \cite{brezzi1994three}, where the space of the three Lagrange multipliers used to weakly impose the continuity of the solution is (a priori) independent of the spaces defined on the adjacent domains. As the functional spaces in the subdomains are mutually independent, our choice of basis functions is well-suited for the coupling of solutions obtained on non-conforming (at the interfaces) meshes, with finite element spaces with different polynomial degrees, or with different numerical methods, e.g. finite element method, spectral element method, or isogeometric analysis \cite{cottrell2009isogeometric,hughes2005isogeometric}.

The paper is structured as follows. In Section~\ref{sec:method}, we present the method on an elliptic problem defined over a domain partitioned into two regions. Section~\ref{sec:discretization} focuses on the discretization of the weak formulation derived in Section~\ref{sec:method}. In Section~\ref{sec:relationship}, we briefly compare our method with other non-conforming methods, namely the mortar method, INTERNODES, and the three-field method, and focus on the similarities and peculiarities with respect to our approach. In Section~\ref{sec:infsup} we address the matter of the stability of the method, which is strictly related to the inf-sup condition. In Section~\ref{sec:applications}, the method is used to solve two-dimensional benchmark problems with finite element discretizations in the subdomains: the Poisson problem on two subdomains (Section~\ref{subsec:poisson}) and the Navier-Stokes equations on five subdomains (Section~\ref{sec:nstokes}). Finally, in Section~\ref{sec:conclusions} some conclusions are drawn.

\subsection{Notation}
\label{subsec:notation}
The notation adopted in this paper is standard and commonly found in the literature; see e.g. \cite{quarteroni2008numerical}.
Given a generic open and bounded domain $\Omega$ embedded in $\mathbb R^d$, we define, for all $\varphi, \psi: \Omega \rightarrow \mathbb{R}$ and all $\boldsymbol \varphi, \boldsymbol \psi: \Omega \rightarrow \mathbb{R}^d$
\begin{equation}
(\varphi,\psi)_{\Omega} := \int_{\Omega} \varphi \psi \, \text d \textbf x, \quad (\boldsymbol \phi, \boldsymbol \psi)_{\Omega} := \int_{\Omega} \boldsymbol \varphi \cdot \boldsymbol \psi\, \text{d} \mathbf x,
\end{equation}
and consider the following Hilbert spaces
\begin{align}
L^2(\Omega) &:= \{ \varphi: \Omega \mapsto \mathbb{R}: (\varphi, \varphi)_\Omega \ < \infty \},\\
H^1(\Omega) &:= \{ \varphi \in L^2(\Omega): \nabla \varphi \in [L^2(\Omega)]^d\}, \\
H(\text{div};\Omega) &:= \{ \boldsymbol \phi \in [L^2(\Omega)]^d: \text{div} \boldsymbol \phi \in L^2(\Omega)\},
\end{align}
with the associated norms
\begin{align}
\Vert \varphi \Vert_{L^2(\Omega)}^2&:=(\varphi,\varphi)_\Omega, \\
\Vert \varphi \Vert_{H^1(\Omega)}^2&:=(\varphi,\varphi)_\Omega + (\nabla \varphi, \nabla \varphi)_\Omega, \\
 \Vert \boldsymbol \phi \Vert_{H(\text{div};\Omega)}^2&:=(\boldsymbol \phi,\boldsymbol \phi )_{\Omega} + (\text{div}\boldsymbol \phi,\text{div}\boldsymbol \phi)_{\Omega}.
\end{align}
Given a measurable set $\Sigma \subseteq \partial \Omega$ (where $\partial \Omega$ denotes the boundary of $\Omega$), we also define
\begin{equation}
 H^1_{\Sigma}(\Omega) := \{\varphi \in H^1(\Omega):\, \varphi = 0 \text{ on } \Sigma \}.
\end{equation}

Under the assumption of sufficient regularity of $\Sigma$, there exists a unique linear and continuous application $\gamma_{\Sigma}: H^1(\Omega) \mapsto L^2(\Sigma)$ called trace operator \cite{quarteroni2014numerical, salsa2016partial} such that $\gamma_{\Sigma} \varphi = \varphi|_{\Sigma}$ for all $\varphi \in H^1(\Omega) \cap C^0(\bar{\Omega})$, having indicated with $C^0(\bar{\Omega})$ the space of continuous functions over the closure of $\Omega$. The range of such operator is denoted $H^{1/2}(\Sigma) \subset L^2(\Sigma)$. We recall that
\begin{equation}
\Vert \eta \Vert_{H^{1/2}(\Sigma)} : = \inf_{\substack{\varphi \in H^1(\Omega) \\ \varphi|_{\Sigma} = \eta}} \Vert \varphi \Vert_{H^1(\Omega)}
\end{equation}
is a norm for $H^{1/2}(\Sigma)$ \cite{braess1999multigrid}.

For each linear Hilbert space $\mathcal H$, we denote $\mathcal H'$ the space of linear and bounded functionals on $\mathcal H$, namely its dual space. In particular, we adopt the notation $H^{-1}(\Omega) := (H^{1}(\Omega))'$ and $H^{-1/2}(\Sigma) := (H^{1/2}(\Sigma))'$. The action of an element of the dual space $\xi \in \mathcal H'$ on an element of the Hilbert space $\varphi \in \mathcal H$ is indicated $\langle \xi, \varphi \rangle_{\mathcal H'}$ or simply $\langle \xi, \varphi \rangle$ whenever ambiguity does not arise. Furthermore, we will simply indicate $\langle \xi, \varphi \rangle_{\Sigma}$ the duality in $H^{-1/2}(\Sigma)$. We note that, for $\varphi \in H^1(\Omega)$, we will adopt the abuse of notation $\langle \xi, \varphi \rangle_{\Sigma}$ with $\xi \in H^{-1/2}(\Sigma)$ to indicate the duality of $\xi$ with the trace of $\varphi$ on $\Sigma$. Even though not mathematically rigorous, this notation is commonly used in the literature.
\section{Theory of primal hybrid methods }
\label{sec:method}
In this section, we recall the theory of primal hybrid methods for the solution of Partial Differential Equations (PDEs). These approaches are based on the \textit{primal hybrid principle} \cite{raviart1977primal}, according to which the continuity across subdomains is weakened by means of Lagrange multipliers. We refer the reader to \cite{boffi2013mixed,acharya2016primal, belgacem1999mortar} for the theory of primal hybrid methods. Here, we recall the main ideas by following the presentation in \cite{braess1999multigrid}. We also restrict ourselves to only two partitions of the domain; however -- as we shall see in Section~\ref{subsec:generalization} -- the method extends to an arbitrary number of partitions.

We are interested in solving a generic PDE described by a second order elliptic operator on an open and bounded domain $\Omega$ with homogeneous Dirichlet boundary conditions on $\partial \Omega$. Specifically, we assume that $a(\varphi,\psi)$  for $\varphi, \psi \in H^1(\Omega)$ is the bilinear form corresponding to the elliptic operator and $f$ is a given forcing term; we consider problems whose weak formulation can be written as:
\begin{weakform}{}
given $f \in H^{-1}(\Omega)$, find $u \in  H^1_{\partial \Omega}(\Omega)$, such that
\begin{equation}
a(u,v) = \langle f,v \rangle \quad \forall v \in  H^1_{\partial \Omega}(\Omega).
\label{eq:fullpoissonweak}
\end{equation}
\label{weak:fullpoisson}
\end{weakform}
In the sequel, we will often use the Poisson problem with homogeneous boundary conditions
\begin{equation}
\begin{alignedat}{3}
-\Delta u &= f\quad && \text{in } \Omega, \\
u & = 0 && \text{on } \partial \Omega,
\label{eq:poisson}
\end{alignedat}
\end{equation}
as representative of this class of problems. In this specific case, $a(\varphi,\psi) = (\nabla \varphi, \nabla \psi)_\Omega$.

Let us assume that the domain $\Omega$ can be partitioned into two non-overlapping open and bounded domains, such that $\Omega = \Omega_1 \cup \Omega_2$ and $\Omega_1 \cap \Omega_2 = \emptyset$; we denote $\Gamma$ the interface between the two domains, i.e. $\Gamma = \overline{\Omega}_1 \cap  \overline{\Omega}_2$. Our goal is to solve, rather than the global problem \weakref{weak:fullpoisson}, two local and coupled problems defined on the partitions $\Omega_i$, such that the global solution can be constructed by combining the solutions of the local problems. To this end, let us introduce the functional spaces $\mathcal X^{(i)} = H^{1}_{\partial \Omega \cap \partial \Omega_i}(\Omega_i)$ and
\begin{equation}
\mathcal X := \{\varphi \in L^2(\Omega): \varphi|_{\Omega_i} \in \mathcal X^{(i)} \text{ for } i = 1,2 \},
\end{equation}
which is a Hilbert space when endowed with the (broken) norm
\begin{equation}
\Vert \varphi \Vert_{\mathcal X}^2 := \sum_{i = 1}^2 \Vert \varphi|_{\Omega_i} \Vert_{ H^1(\Omega_i)}^2.
\label{eq:broken_norm}
\end{equation}
The space $H^1_{\partial \Omega}(\Omega)$ is characterized as a subspace of $\mathcal X$ under suitable conditions \cite{braess1999multigrid} which we will state in the following Lemma and motivate in its proof. We remark that an analogous result is presented in \cite{boffi2013mixed}.
\begin{proposition}{\cite[cf. Proposition 2.1.1]{boffi2013mixed}}
\begin{equation}
H^{1}_{\partial \Omega}(\Omega) \equiv \widetilde{\mathcal V} := \{\varphi \in \mathcal X: \sum_{i = 1}^2 \langle \boldsymbol \phi \cdot \textbf{n}_i, \varphi \rangle_{\partial \Omega_i} = 0 \quad \forall \boldsymbol \phi \in H(\textup{div}; \Omega) \},
\end{equation}
where $\textbf n_i$ is the outward unit vector normal to $\partial \Omega_i$.
\label{proposition:equiv}
\end{proposition}
\begin{proof}
Firstly, let us recall that for any $\boldsymbol \phi \in H^1(\text {div};\Omega)$ and for $i = 1,2$, Green's formula gives \cite[Lemma 2.1.1]{boffi2013mixed}
\begin{equation}
\int_{\Omega_i} \left( \nabla \varphi \cdot \boldsymbol \phi + \varphi \text{div}\boldsymbol \phi \right ) \,\text{d}\textbf x = \langle \boldsymbol \phi \cdot \textbf n_i, \varphi \rangle_{\partial \Omega_i} \quad \forall \varphi \in \mathcal X. \label{eq:greens}
\end{equation}
Eq.~\eqref{eq:greens} implies that, since the left hand side is bounded, the duality of $\boldsymbol \phi \cdot \textbf n_i \in H^{-1/2}(\partial \Omega_i)$ on the trace of $\varphi|_{\Omega_i}$ is well-defined.

Clearly $\widetilde{\mathcal V} \subset H^1_{\partial \Omega}(\Omega)$, because for all $\varphi \in \widetilde{\mathcal V} \subset L^2(\Omega)$
\begin{equation}
\int_{\Omega} |\nabla \varphi |^2\, \text d \textbf x = \sum_{i = 1}^2 \int_{\Omega_i} | \nabla \varphi |^2\, \text d \textbf x  < \infty,
\end{equation}
and $\varphi = 0$ on $\partial \Omega$ in $H^{1/2}(\partial \Omega)$. Let us show the other inclusion. For every $\varphi \in H^1_{\partial \Omega}(\Omega)$, it holds that $\varphi \in L^2(\Omega)$ and $\varphi|_{\Omega_i} \in \mathcal X^{(i)}$  for $i = 1,2$, which implies that $H^1_{\partial \Omega}(\Omega) \subset \mathcal X$. Moreover, by applying Green's formula as in Eq.~\eqref{eq:greens}, we find for all $\boldsymbol \phi \in H(\text{div};\Omega)$
\begin{equation}
\sum_{i = 1}^2 \langle \boldsymbol \phi \cdot \textbf{n}_i, \varphi \rangle_{\partial \Omega_i} =  \sum_{i = 1}^2 \int_{\Omega_i} \left( \nabla \varphi \cdot \boldsymbol \phi + \varphi \text{div}\boldsymbol \phi \right ) \,\text{d}\textbf x = \int_{\Omega} \left( \nabla \varphi \cdot \boldsymbol \phi + \varphi \text{div}\boldsymbol \phi \right ) \,\text{d}\textbf x = \langle \boldsymbol \phi \cdot \textbf n , \varphi \rangle_{\partial \Omega} = 0,
\label{eq:equivalenceproof}
\end{equation}
where  the last equality comes from the fact that $\varphi$ has null trace on the boundary $\partial \Omega$. Since Eq.~\eqref{eq:equivalenceproof} shows that $H^1_{\partial \Omega}(\Omega) \subset \widetilde{\mathcal V}$, it must be $H^1_{\partial \Omega}(\Omega) \equiv \widetilde{\mathcal V}$.
\end{proof}
\begin{remark}
If the PDE is equipped with Dirichlet conditions on $\partial \Omega_D$ and Neumann conditions on $\partial \Omega_N$, the space $\mathcal X$ must be defined such that $\varphi|_{\Omega_i}$ belongs to $\mathcal X^{(i)}$ for $i = 1,2$. In this case it is not sufficient to ask that $\boldsymbol \phi \in H(\text{div};\Omega)$ in the definition of $\mathcal V$ to have the equivalence between $H^1_{\partial \Omega} (\Omega)$ and $\mathcal V$. In particular, the space $H(\text{div};\Omega)$ must be restricted to functions $\boldsymbol \phi$ such that $\langle \boldsymbol \phi \cdot \textbf n, \varphi \rangle_{\partial \Omega} = 0$ for each $\varphi \in \mathcal V$.
\end{remark}
The condition $\sum_{i = 1}^2 \langle \boldsymbol \phi \cdot \textbf n_i, \varphi \rangle_{\partial \Omega_i} = 0$ for each $\boldsymbol \phi \in H(\text{div};\Omega)$ is global, in the sense that it involves the trace of $\varphi|_{\Omega_i}$ on the whole $\partial \Omega_i$, even though it essentially constrains the restrictions of $\varphi$ to $\Omega_1$ and $\Omega_2$ to have the same trace at the common interface $\Gamma$. Unfortunately, splitting the dualities into two parts corresponding to $\Gamma$ and $\partial \Omega \setminus \Gamma$ is not allowed, as the restrictions of the traces to portions of $\partial \Omega_i$ can lead to unbounded dualities. To overcome this issue, we introduce
\begin{equation}
H^{1/2}_{00}(\Gamma_i) := \{ \eta \in H^{1/2}(\Gamma_i): E_0^{(i)} \eta \in H^{1/2}(\partial \Omega_i) \},
\end{equation}
with norm
\begin{equation}
\Vert \eta \Vert_{H^{1/2}_{00}(\Gamma_i)}:= \Vert E_0^{(i)} \eta \Vert_{H^{1/2}(\partial \Omega_i)},
\end{equation}
where $E_0^{(i)} \eta$ is the trivial extension by zero of $\eta$ to the whole boundary of $\partial \Omega_i$ and $\Gamma_i = \partial \Omega_i \setminus \partial \Omega$.  In the following, we will consider $H^{1/2}_{00}(\Gamma) : = H^{1/2}_{00}(\Gamma_1) \cap H^{1/2}_{00}(\Gamma_2)$. Let us define the spaces
\begin{equation}
\mathcal X_{00} := \{ \varphi \in \mathcal X: [\varphi]_{\Gamma} \in H^{1/2}_{00}(\Gamma)\},
\end{equation}
where $[\varphi]_{\Gamma}$ denotes by our convention the difference of the traces of $\varphi|_{\Omega_2}$ and $\varphi|_{\Omega_1}$ on $\Gamma$, and
\begin{equation}
\Lambda :=  H^{-1/2}_{00}(\Gamma),
\label{eq:lambdadef}
\end{equation}
with norm
\begin{equation}
\Vert \eta \Vert_{\Lambda} := \Vert \eta \Vert_{H^{-1/2}_{00}(\Gamma_1)} + \Vert \eta \Vert_{H^{-1/2}_{00}(\Gamma_2)}.
\end{equation}
Furthermore, we introduce the bilinear form
\begin{equation}
b(\varphi, \xi) := \langle \xi, [\varphi]_\Gamma \rangle_{\Lambda}
\end{equation}
for $\varphi \in \mathcal X_{00}$ and $\xi \in \Lambda$. It can be easily verified \cite{braess1999multigrid} that another characterization of $H^1_{\partial \Omega}(\Omega)$ analogous to that in Proposition~\ref{proposition:equiv} is given by
\begin{equation}
H^1_{\partial \Omega} (\Omega) \equiv \mathcal V := \{\varphi \in \mathcal X_{00}: b(\varphi, \xi) = 0\quad \forall \xi \in \Lambda \}.
\label{eq:char}
\end{equation}
In the sequel, we will use the letter $\mathcal V$ to refer to $H^1_{\partial \Omega}(\Omega)$.

We are now ready to state the primal hybrid formulation of the original weak formulation \weakref{weak:fullpoisson}. We remark that, whenever applied to functions of $\mathcal X$, the bilinear form $a(\cdot,\cdot)$ is to be intended as the sum of the bilinear forms restricted to the two subdomains.
\begin{weakform}{}
given $f \in H^{-1}(\Omega)$, find $u \in  \mathcal X_{00}$ and $\lambda \in \Lambda$ such that
\begin{equation}
\begin{alignedat}{3}
&a(u,v) + b(v,\lambda) = \langle f,v \rangle&&\quad \forall  v \in \mathcal X_{00}, \\
&b(u,\eta) = 0 && \quad \forall \eta \in \Lambda.
\end{alignedat}
\label{eq:saddleweak}
\end{equation}
\label{weak:saddleweak}
\end{weakform}
\begin{proposition}
If $u \in \mathcal V$ is a solution of \weakref{weak:fullpoisson} and there exists $\lambda \in \Lambda$ such that
\begin{equation}
b(v,\lambda) = \langle f, v \rangle - a(u,v)\quad \forall v \in \mathcal X_{00},
\label{eq:lambdacondition}
\end{equation}
then $(u,\lambda) \in \mathcal X_{00} \times \Lambda$ is a solution of \weakref{weak:saddleweak}. On the other hand, if $(u,\lambda) \in \mathcal X_{00} \times \Lambda$ is a solution of \weakref{weak:saddleweak}, then $u \in \mathcal V$ and $u$ is a solution of \weakref{weak:fullpoisson}.
\label{theorem:eq_bil_forms}
\end{proposition}
\begin{proof}
Let $u \in \mathcal V$ be a solution of \weakref{weak:fullpoisson}, then $u \in \mathcal X_{00}$ and the second condition in Eq.~\eqref{eq:saddleweak} is satisfied because of the definition \eqref{eq:char}. The first condition in Eq.~\eqref{eq:saddleweak} is satisfied when choosing $\lambda \in \Lambda$ such that Eq.~\eqref{eq:lambdacondition} is verified. Conversely, if $(u,\lambda) \in \mathcal X_{00} \times \Lambda$ is a solution for \weakref{weak:saddleweak}, then $u \in \mathcal V$ because of the second condition in Eq.~\eqref{eq:saddleweak}. Moreover, for each $v \in \mathcal V$, $b(v,\xi) = 0$ for all $\xi \in \Lambda$ and, in particular, for $\xi = \lambda$, and the first condition in Eq.~\eqref{eq:saddleweak} becomes Eq.~\eqref{eq:fullpoissonweak}.
\end{proof}
\begin{remark}
If we consider the Poisson equation~\eqref{eq:poisson}, then Eq.~\eqref{eq:lambdacondition} is verified by taking  $\boldsymbol \phi = - \nabla u$ and by choosing $\lambda \in \Lambda$ such that $\boldsymbol \phi \cdot \textbf n_1 = \lambda$, $\mathbf n_1$ being the outward unit vector normal to $\partial \Omega_1$. Indeed, by using integration by parts we find for all $v \in \mathcal X_{00}$
\begin{align}
b(v,\lambda) &= \langle \boldsymbol \phi \cdot \textbf n_1, [v]_{\Gamma} \rangle_{\Lambda} = - \sum_{i=1}^2 \int_{\Gamma}  \nabla u \cdot \textbf n_i v \, \text{d}\textbf{s} \\
&= \sum_{i=1}^2 \left (\int_{\Omega_i} fv\,\text{d}\textbf{x} - \int_{\Omega_i} \nabla u \cdot \nabla v\, \text{d}\textbf{x} \right )= \langle f,v \rangle - a(u,v),
\end{align}
where we used the fact that $\textbf n_1 = -\textbf n_2$. Note that, if we defined the jump across the interface of a function $\varphi \in \mathcal X_{00}$ as the difference of the traces on $\Gamma$ of $\varphi|_{\Omega_1}$ and $\varphi|_{\Omega_2}$, then  $\boldsymbol \phi \cdot \textbf n_2 = \lambda$. Hence, the Lagrange multiplier in Eq.~\eqref{eq:saddleweak} plays the role of the normal derivative of $u$ at the interface $\Gamma$ \cite{wohlmuth2000mortar}, with the direction of the normal at the interface being determined by the definition of the jump.
\label{remark:derivative}
\end{remark}

\section{Discretization of the primal hybrid formulation}
\label{sec:discretization}
We now consider the discretization of the weak formulation \weakref{weak:saddleweak}.
We take two arbitrary finite dimensional functional spaces $\mathcal X^{h,(1)} \subset \mathcal X^{(1)}$ and $\mathcal X^{h,(2)}\subset \mathcal X^{(2)}$ spanned by two sets of basis functions $\varphi^{(1)}_i \in \mathcal X^{(1)}$ (with $i = 1,\ldots,n_\text{bf}^{(1)}$) and $\varphi^{(2)}_i \in \mathcal X^{(2)}$ (with $i = 1,\ldots,n^{(2)}_\text{bf}$)  respectively. We assume that functions in $\mathcal X^{h,(1)}$ and $\mathcal X^{h,(2)}$ can be trivially extended by zero in the other domain and that such extension belong to $\mathcal X_{00}$. The discrete version of the global space $\mathcal X_{00}$ is consequently obtained by considering the space $\mathcal X^h \subset \mathcal X_{00}$ of dimension $\text{dim}(\mathcal X^h) = n_{\text{bf}} = n^{(1)}_{\text{bf}} + n^{(2)}_{\text{bf}}$ and spanned by the basis functions
\begin{equation}
\{\varphi_i\}_{i = 1}^{n_{\text{bf}}} = \{ \varphi^{(1)}_i\}_{i = 1}^{n^{(1)}_{\text{bf}}} \cup \{ \varphi^{(2)}_i\}_{i = 1}^{n^{(2)}_\text{bf}}.
\end{equation}
The solution can be then approximated as $u \approx u^h = \sum_{i = 1}^{n_\text{bf}} u_{i} \varphi_i$. In the numerical applications in Section~\ref{sec:applications}, we will consider standard finite element Lagrangian basis functions built over suitable triangulations $\mathcal T^{h,(1)}$ and $\mathcal T^{h,(2)}$ of $\Omega_1$ and $\Omega_2$ respectively for the discretization of $\mathcal{X}^{(1)}$ and $\mathcal{X}^{(2)}$; we will always assume that such triangulations meet standard regularity requirements \cite{quarteroni2008numerical}, but we do not require the conformity of the global mesh $\mathcal T^h = \mathcal T^{h,(1)} \cup \mathcal T^{h,(2)}$. We define conforming meshes those meshes for which the intersection of two elements is either null, a vertex or a whole edge; in non-conforming meshes, on the contrary, two elements can also share portions of their edges. The discretization parameter $h$ is generic and defines a family of discretized spaces; when using finite elements, for example, $h$ refers to the maximum edge length of an element -- often called mesh size -- in the triangulations of $\Omega_1$ and $\Omega_2$. More generally, $h$ could be also considered a characteristic of the single subdomain, since -- as we already mentioned -- the discretizations in $\Omega_1$ and $\Omega_2$ are independent one of the other and could be obtained from different discretization methods (e.g. finite elements for $\Omega_1$ and isogeometric analysis for $\Omega_2$).

Our proposition is to discretize $\Lambda$ as $\Lambda^\delta$ by using a set of basis functions $\xi_i \in \Lambda$, such that $\lambda \in \Lambda$ is approximated as $\lambda \approx \lambda^\delta =\sum_{i = 1}^{n_\Gamma} \lambda_{i} \xi_i$. We remark that we characterize the refinement levels for $\mathcal X^{h,(1)}$, $\mathcal X^{h,(2)}$ and $\Lambda^\delta$ with different discretization parameters $h$ and $\delta$: this is to indicate that the discretization of $\Lambda$ is indeed independent of the discretization on $\Omega_1$ and $\Omega_2$. For instance, in the two-dimensional case, a suitable choice would consist of choosing as $\xi_i$ the basis functions associated to the low-frequencies of the Fourier basis defined on the common interface $\Gamma$, and the accuracy of the discretization of $\Lambda^\delta$ can be increased independently of $h$ by adding Fourier basis functions to the set $\xi_i$. In the numerical simulations of Section \ref{sec:applications} we will follow this approach. Alternative  possibilities for the discretization of the Lagrange multiplier space include other spectral basis functions, such as e.g. Legendre or Chebyshev polynomials.


The discrete space for the approximation of $\mathcal V$ is then defined as
\begin{equation}
\mathcal V^{h,\delta} := \{ \varphi^h \in \mathcal X^h: b(\varphi^h, \xi^\delta)= 0\quad \forall \xi^\delta \in \Lambda^\delta \},
\end{equation}

\begin{remark} $\mathcal V^{h,\delta}$ is not a subspace of $\mathcal V$. As a matter of fact, if $\Lambda^\delta$ is not equal to $\Lambda$, then there may exist $\xi \in \Lambda$, $\xi  \not \in \Lambda^\delta$ such that $b(\varphi^h, \xi) \neq 0$ for some $\varphi^h \in \mathcal V^{h,\delta}$, and therefore $\varphi^h \not \in \mathcal V$. If we replaced $\mathcal V$ by $\mathcal V^{h,\delta}$ in \weakref{weak:fullpoisson}, we would obtain a \textit{non-conforming method}, i.e. a numerical method in which the discretized search space is not contained into the continuous search space. The generalized version of Cea's lemma for this family of methods is Strang's second lemma \cite{ciarlet2002finite}, which states that the solution $u^h$ of the discretized version of \weakref{weak:fullpoisson} satisfies
\begin{equation}
\Vert u - u^h \Vert_{\mathcal V^{h,\delta}} \leq C \left ( \inf_{v^h \in \mathcal V^{h,\delta}} \Vert u - v^h \Vert_{\mathcal{V}^{h,\delta}} + \sup_{w^h \in \mathcal{V}^{h,\delta}} \dfrac{| a(u,w^h) - \langle f,w^h \rangle |}{\Vert w^h \Vert_{\mathcal V^{h,\delta}}} \right ),
\label{eq:strang}
\end{equation}
where $C>0$ and $\Vert \cdot \Vert_{\mathcal V^{h,\delta}}$ is a norm for $\mathcal V^{h,\delta}$. Note that the consistency error -- i.e. the second term of the right hand side in Eq.~\eqref{eq:strang} -- is identically zero for each $w^h \in \mathcal V^{h,\delta}$ if $\mathcal V^{h,\delta} \subset V$ because $u$ is a solution of \weakref{weak:fullpoisson}.
\end{remark}

The discretization of \weakref{weak:saddleweak} is simply obtained by replacing the continuous functional spaces with their discrete counterparts, namely:
\begin{weakform}{}
given $f \in H^{-1}(\Omega)$, find $u^h \in  \mathcal X^h$ and $\lambda^\delta \in \Lambda^\delta$ such that
\begin{equation}
\begin{alignedat}{3}
&a(u^h,v^h) + b(v^h,\lambda^\delta) = \langle f,v^h \rangle &&\quad \forall v^h \in \mathcal X^h, \\
&b(u^h,\eta^\delta) = 0 &&\quad \forall \eta^\delta \in \Lambda^\delta.
\end{alignedat}
\label{eq:saddleweak_h}
\end{equation}
\label{weak:saddleweak_h}
\end{weakform}

By expanding $u^h$ and $\lambda^\delta$ on their respective bases, Eq.~\eqref{eq:saddleweak_h} can be rewritten in system form as
\begin{equation}
\begin{bmatrix}
A & B^T \\
B & 0
\end{bmatrix}
\begin{bmatrix}
\textbf u \\
\boldsymbol \lambda
\end{bmatrix}
=
\begin{bmatrix}
\textbf f \\
\textbf 0
\end{bmatrix},
\label{eq:algebraic}
\end{equation}
where $A_{ij} = a(\varphi_j, \varphi_i)$, $B_{ij} = b(\varphi_j, \xi_i)$, $\textbf u_i = u_i$, $\boldsymbol \lambda = \lambda_i$ and $\textbf f_i = \langle f, \varphi_i \rangle$. By arranging the basis functions $\varphi_n$ and the degrees of freedom such that all the basis functions corresponding to $\Omega_1$ come before those of $\Omega_2$, system \eqref{eq:algebraic} can be written as
\begin{equation}
\begin{bmatrix}
A_1 & 0 & -B_1^T \\
0     & A_2 & B_2^T \\
-B_1 & B_2 & 0 \\
\end{bmatrix}
\begin{bmatrix}
\textbf u_1 \\
\textbf u_2 \\
\boldsymbol \lambda
\end{bmatrix}
=
\begin{bmatrix}
\textbf f_1 \\
\textbf f_2 \\
\textbf 0
\end{bmatrix},
\label{eq:algebraic_div}
\end{equation}
where $(B_1)_{ij} = \int_{\Gamma} \varphi^{(1)}_j \xi_i\, \text d\textbf{s}$ and $(B_2)_{ij} = \int_{\Gamma} \varphi^{(2)}_j \xi_i \, \text d\textbf{s}$ are coupling matrices. Clearly, $B_1$ and $B_2$ are likely to be sparse, as only the basis functions $ \varphi^{(1)}_i$ and  $ \varphi^{(2)}_i$ not vanishing on $\Gamma$ lead to non-zero integrals.

In this paper, the computation of the coupling matrices is performed by numerically integrating by Gauss quadrature rules \cite{quarteroni2010numerical} the integrals. Let us consider for instance the case of $B_1$ in the two dimensional case. The triangulation $\mathcal T^{h,(1)}$ induces on $\Gamma$ a partition into $n_{\text{el},\Gamma}^{(1)}$ elements, i.e. $\Gamma = \bigcup_{i = 1}^{n_{\text{el},\Gamma}^{(1)}} E_i^{(1)}$.  Given a Gauss quadrature rule of order $2q-1$, the approximation of each term of $B_1$ is computed as
\begin{equation}
(B_1)_{mn} = \int_{\Gamma} \varphi^{(1)}_n \xi_m\, \text d\textbf{s} = \sum_{i = 1}^{n_{\text{el},\Gamma}^{(1)}} \int_{\Gamma} \varphi^{(1)}_n \xi_m\, \text d\textbf{s} \approx  \sum_{i = 1}^{n_{\text{el},\Gamma}^{(1)}} \sum_{j = 1}^{q}  |\text{det}(J_i)|\varphi^{(1)}_n (\phi_i(\mathbf x_{j}^\text{gq})) \xi_m (\phi_i(\mathbf x_{j}^\text{gq})) \omega_{j},
\label{eq:apprB}
\end{equation}
where $\text{det}(J_i)$ is the determinant of the Jacobian of the map $\phi_i: E_i^{(1)} \rightarrow (-1,1)$ from $E_i^{(1)}$ to the reference interval $(-1,1)$, $\mathbf x_{j}^\text{gq}$ is the $j^\text{th}$ Gauss quadrature node in $(-1,1)$ and $\omega_{j}$ is the associated weight. As it is evident from Eq.~\eqref{eq:apprB}, in order to compute the approximation of $B_1$ it is sufficient to being able to evaluate the product $\varphi^{(1)}_n (\phi_i(\mathbf x_{j}^\text{gq})) \xi_m (\phi_i(\mathbf x_{j}^\text{gq}))$ at each quadrature node.
\subsection{Generalization to multiple subdomains}
\begin{figure}
\centering
\begin{minipage}{0.45\textwidth}
\resizebox{0.8\textwidth}{!}{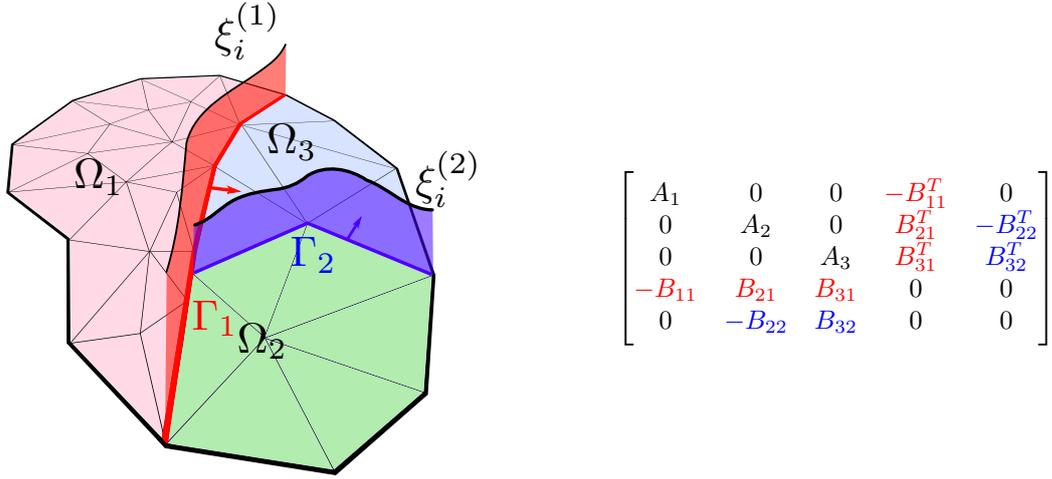}
\end{minipage}
\begin{minipage}{0.45\textwidth}
\begin{equation*}
\begin{bmatrix}
A_1 & 0 & 0 & \textcolor{red}{-B_{11}^T} & 0 \\
0     & A_2 & 0 &  \textcolor{red}{B_{21}^T} & \textcolor{blue}{ -B_{22}^T} \\
0 & 0 & A_3 &  \textcolor{red}{B_{31}^T} &  \textcolor{blue}{B_{32}^T}\\
 \textcolor{red}{-B_{11}} &  \textcolor{red}{B_{21}} &  \textcolor{red}{B_{31}} & 0 & 0 \\
0 &  \textcolor{blue}{-B_{22}} &  \textcolor{blue}{B_{32}} & 0 & 0 \\
\end{bmatrix}
\end{equation*}
\end{minipage}
\caption{On the left, example of a three-way partition of $\Omega$ with two interfaces; each interface corresponds to a Lagrange multiplier space. On the right, corresponding system matrix obtained from the discretization of the spaces.}
\label{fig:example_partition}
\end{figure}
\label{subsec:generalization}
In the previous sections we decided to limit ourselves to the case where the domain of the PDE  is partitioned into two subdomains. This choice is motivated mainly by the fact that considering the generic case of multiple subdomains leads necessarily to complexity in the notation. We refer the reader to the already mentioned references \cite{braess1999multigrid,belgacem1999mortar} for examples of how the functional spaces we considered in Section~\ref{sec:method} could be adapted to the case of multiple subdomains.
One aspect that differentiates our approach from other methods (such as the mortar method) is that, in the discretization process, our method requires to define a set of basis functions for the Lagrange multiplier space of each interface. These bases can be chosen independently one from the other.

As an example, Fig.~\ref{fig:example_partition} (left) shows a three-way partition of $\Omega$ into three domains with two interfaces. Each of the interfaces $\Gamma_1$ and $\Gamma_2$ requires the definition of a corresponding space for the Lagrange multipliers. After the discretization, the matrix of the algebraic system can be written as displayed in  Fig.~\ref{fig:example_partition}~(right), where the matrices $(B_{ij})_{mn} = \int_{\Gamma_{j}} \varphi^{(i)}_n \xi^{(j)}_m\,\text{d}\textbf{s}$ discretize the coupling between the $i^\text{th}$ domain and the $j^{\text{th}}$ interface. We remark that the signs of the coupling matrices are determined by the definition of the normals at each interface. Notice that, besides the choice of the orientation of the normals at the interfaces, there is no hierarchy among the subdomains.
\section{Relationship with other non-conforming methods}
\label{sec:relationship}
\subsection{Relationship with the mortar method}
\label{subsec:relationmortar}
The mortar method can be derived from the same problem written in primal hybrid formulation \weakref{weak:saddleweak} we considered in Section~\ref{sec:method} \cite{wohlmuth2000mortar}. Its discretized weak formulation could be rewritten in the form of a saddle-point problem similar to Eq.\eqref{eq:saddleweak} in which the space of Lagrange multiplier $\Lambda^h_{M}$ depends on the discretization of either $\Omega_1$ or $\Omega_2$; see \cite{seshaiyer1998convergence,seshaiyer2003stability}. In particular, the classic mortar method requires assigning to $\Omega_1$ or to $\Omega_2$ the role of master and slave domains. The basis functions of $\Lambda^h_{M}$ are chosen as the trace of the basis functions defined over the triangulation of the slave domain which do not vanish on $\Gamma$; the polynomial order of the basis functions on the extrema of $\Gamma$ is usually decreased by one.

With respect to the mortar method, we believe that the main advantages of our approach are the following
\begin{enumerate}
\item{
the accuracy of the coupling can be increased or decreased by varying the number of Lagrange multipliers at the interfaces independently of the discretization in the subdomains;
}
\item{
the solution is independent of the partition of the subdomains into master and slaves;
}
\item{
the computation of the coupling matrices does not require projections between meshes, which makes the implementation of the method easier.
}
\end{enumerate}
One drawback of our method is that the Lagrange multiplier space has to be rich enough to provide the necessary accuracy, but coarse enough to satisfy the inf-sup condition as described in Section~\ref{sec:infsup}.

\subsection{Relationship with INTERNODES}
The INTERNODES (INTERpolation for NOnconforming DEcompositionS) method \cite{deparis2016internodes, forti2016parallel} is based on an interpolation approach, rather than the $L^2$-projection approach which characterizes the mortar method. Given each interface, the two adjacent subdomains are given the role of master and slave domains. Similarly to the mortar method, the traces of the (finite element or spectral element) basis functions defined over the meshes of the master and slave domains are used to enforce the continuity of the solution and the normal stresses. More precisely, two interpolation operators -- or intergrid operators -- are defined: the interpolation operator from the master to the slave domain is used to ensure the continuity of the solution, while the interpolation operator from the slave to the master domain enforces the continuity of the normal fluxes. INTERNODES has been proven to retain the optimal convergence properties of the mortar method. For more information about the method and its analysis, we refer the reader to \cite{gervasio2016analysis}.

Being INTERNODES closely related to the mortar method, compared to the former our approach offers the same advantages we presented in Section~\ref{subsec:relationmortar} except for 3. Indeed, INTERNODES has the big advantage of being simple to implement and allowing for small geometric non-conformity. We believe that our method is as simple to implement as INTERNODES and that it can be extended to non-conforming geometries with the help of localized Rescaled Radial Basis Interpolation  \cite{forti2016parallel,deparis2014rescaled}. One complexity of INTERNODES comes from the special treatment of integrals at the intersection of the interface $\Gamma$ with portions of the boundary where non-homogeneous Neumann conditions are imposed. The method proposed here does not need such special treatment.
\subsection{Relationship with the three-field method}
The three-field method was originally proposed in \cite{brezzi1994three} and analyzed in \cite{brezzi2001error}. Compared to the mortar method, it has had significantly less impact on the domain decomposition community.

The multidomain extension of the weak formulation \weakref{weak:fullpoisson} by the three fields method reads \cite{quarteroni1999domain}:
\begin{weakform}{}
for $i = 1,2,$ find  $u^{(i)} \in  \mathcal X^{(i)}$, $\sigma^{(i)} \in H^{-1/2}(\Gamma)$ and $\psi \in H^{1/2}(\Gamma)$ such that
\begin{equation}
\begin{alignedat}{3}
&a(u^{(1)},v^{(1)}) - \langle \sigma^{(1)}, v^{(1)} \rangle_{H^{-1/2}(\Gamma)} = \langle f, v^{(1)} \rangle && \quad \forall v^{(1)} \in \mathcal X^{(1)}, \\
& \langle \eta^{(1)}, \psi - u^{(1)} \rangle_{H^{-1/2}(\Gamma)} = 0 &&\quad \forall \eta^{(1)} \in H^{-1/2}(\Gamma), \\
& \langle \sigma^{(1)} + \sigma^{(2)}, \rho \rangle_{H^{-1/2}(\Gamma)} = 0 &&\quad \forall \rho \in H^{1/2}(\Gamma),\\
& \langle \eta^{(2)}, \psi - u^{(2)} \rangle_{H^{-1/2}(\Gamma)} = 0 &&\quad \forall \eta^{(2)} \in H^{-1/2}(\Gamma),\\
&a(u^{(2)},v^{(2)}) - \langle \sigma^{(2)}, v^{(2)} \rangle_{H^{-1/2}(\Gamma)} = \langle f, v^{(2)} \rangle && \quad \forall v^{(2)} \in \mathcal X^{(2)}.
\end{alignedat}
\label{eq:threefield}
\end{equation}
\label{weak:threefield}
\end{weakform}
It can be proven (\cite[Proposition 1.7.1]{quarteroni1999domain}) that if $u$ is the solution of \weakref{weak:fullpoisson} and  $u^{(i)}$, $\sigma^{(i)}$, $\psi$ are solutions of \weakref{weak:threefield}, then $u^{(i)} = u|_{\Omega_i}$, $\sigma^{(i)} = (\nabla_L  u \cdot \mathbf n_i)_\Gamma$ (where $\nabla_L u \cdot \mathbf n_i$ indicates the conormal derivative of $u$ with respect to the normal vector $\mathbf n_i$), and $\psi = u|_{\Gamma}$.

The weak formulation \weakref{weak:saddleweak} we derived in Section~\ref{sec:method} can be interpreted as a particular case of \weakref{weak:threefield}. Indeed, let us firstly restrict the space $H^{1/2}(\Gamma)$ to its embedded subset $H^{1/2}_{00}(\Gamma)$ and let us consider the particular case in which $\lambda = \sigma^{(1)} = - \sigma^{(2)} \in \Lambda = H^{-1/2}_{00}(\Gamma)$: then, the third equation in Eq.~\eqref{eq:threefield} is automatically satisfied for all choices of $\rho \in H^{1/2}_{00}(\Gamma)$. Moreover, subtracting the second and fourth equations evaluated at the same $\eta^{(1)} = \eta^{(2)} = \eta \in \Lambda$ yields
\begin{equation}
\langle \eta, \psi - u^{(1)} \rangle_\Lambda - \langle \eta, \psi - u^{(2)} \rangle_ \Lambda = \langle \eta, u^{(2)}- u^{(1)}  \rangle _\Lambda.
\label{eq:bnew}
\end{equation}
Obviously, this duality is well defined only if $u^{(2)}- u^{(1)}$ belongs to $H^{1/2}_{00}(\Gamma)$. We, therefore, set $u \in \mathcal X_{00}$ such that $u^{(1)} = u|_{\Omega_1}$ and $u^{(2)} = u|_{\Omega_2}$; Eq.~\eqref{eq:bnew} can be then rewritten as $b(u, \eta) = 0$ for all $\eta \in \Lambda$, i.e. the second equation in  Eq.~\eqref{eq:saddleweak}. The first equation in Eq.~\eqref{eq:saddleweak} is found by adding the first and last equations in \eqref{eq:threefield} tested for all $v \in \mathcal X_{00}$ such that $v^{(1)} = v|_{\Omega_1}$, $v^{(2)} = v|_{\Omega_2}$; observe that also in this case it is necessary to restrict the search space for $v$ to $\mathcal X_{00}$, in order to ensure the well-posedness of $b(v,\lambda) = \langle \lambda, v^{(2)} - v^{(1)} \rangle_{\Lambda}$.

Although \weakref{weak:saddleweak} and \weakref{weak:threefield} are equivalent, their discretizations are not. Indeed, in the three-field method, it is necessary to define the discretizations of the variational spaces of $\sigma^{(1)}$, $\sigma^{(2)}$ and $\psi$.  In contrast, when discretizing W2, the third equation of W4 is not approximated but solved exactly and the second and fourth are merged into a single equation. As we have shown, setting $\sigma^{(1)} = -\sigma^{(2)}$ is efficient because it allows to automatically satisfy the third equation in Eq.~\eqref{eq:threefield}, thus reducing the number of variables. Our approach limits to one the number of spaces to be discretized for each interface, thus allowing better control of the stability of the method.

\section{Inf-sup condition of the discretized problem} \label{sec:infsup}
Problems \weakref{weak:saddleweak} and \weakref{weak:saddleweak_h} are saddle-point problems \cite{quarteroni2014numerical}. As such, their well-posedness depends on the Ladyschenskaja-Babuška-Brezzi inf-sup condition \cite{brezzi1974existence}, which sets the requirements for the uniqueness of the solution as well as the stability of the sequence of problems depending on the discretization parameters (e.g. the mesh size $h$ or the number of basis functions on the interface $n_\Gamma$). We refer the reader to \cite{brezzi1974existence} and \cite{brezzi1990discourse} for a comprehensive description of the inf-sup condition from the functional and algebraic point of view respectively. In this section, we specifically address the well-posedness of \weakref{weak:saddleweak_h}, and we limit ourselves to recall that if the space $\Lambda$ is characterized as in \eqref{eq:lambdadef}, the continuous problem \weakref{weak:saddleweak} has a unique solution \cite{braess1999multigrid}.

Before stating the main stability result for \weakref{weak:saddleweak_h}, we recall that we characterize $a(\cdot,\cdot)$ and $b(\cdot,\cdot)$ as continuous if there exist $\kappa_a > 0$ and $\kappa_b > 0$ such that $a(\varphi,\psi) \leq \kappa_a \Vert \varphi \Vert_\mathcal{X}  \Vert \psi \Vert_\mathcal{X}$ for every $\varphi,\psi \in \mathcal X$ and $b(\varphi, \xi) \leq \kappa_b \Vert \varphi \Vert_{\mathcal X} \Vert \xi \Vert_{\Lambda}$ for every $\varphi \in \mathcal X$, $\xi \in \Lambda$.

The following theorem prescribes the conditions for the well-posedness of \weakref{weak:saddleweak_h}.

\begin{theorem}{\cite[Theorem 3.2]{brezzi1990discourse}}
Assume that $a(\cdot,\cdot)$ and $b(\cdot,\cdot)$ are continuous with constants $k_a>0$ and $k_b>0$, and that there exist $\alpha > 0$ and $\beta > 0$ such that $\mathcal X^h$, $\mathcal V^{h,\delta}$ and $\Lambda^\delta$ satisfy the conditions
\begin{align}
&\inf_{v^h \in \mathcal V^{h,\delta}} \dfrac{a(v^h,v^h)}{\Vert v^h \Vert_{\mathcal{X}}^2} \geq \alpha, \label{eq:condition1}\\
&\inf_{\eta^\delta \in \Lambda^\delta} \sup_{v^h \in \mathcal X^h} \dfrac{b(v^h,\eta^\delta)}{\Vert v^h \Vert_{\mathcal{X}}\Vert \eta^\delta \Vert_{\Lambda}} \geq \beta. \label{eq:condition2}
\end{align}
Then \weakref{weak:saddleweak_h} has a unique solution. Moreover, there exists a constant $C \geq 0$, depending only on $\kappa_a$, $\kappa_b$, $\alpha$ and $\beta$, such that
\begin{equation}
\Vert u - u^h \Vert_{\mathcal X} + \Vert \lambda - \lambda^\delta \Vert_{\Lambda} \leq C \left ( \inf_{v^h \in \mathcal X^h} \Vert u - v^h \Vert_{\mathcal X} + \inf_{\eta^\delta \in \Lambda^\delta} \Vert \lambda - \eta^\delta \Vert_{\Lambda} \right ),
\label{eq:estimate}
\end{equation}
where $(u,\lambda)$ is the solution of \weakref{weak:saddleweak}.
\label{theorem:existence}
\end{theorem}

\subsection{Numerical computation of the inf-sup constant}
The inf-sup condition \eqref{eq:condition2} is satisfied whenever $\Lambda^\delta$ is sufficiently ``small'' compared to $\mathcal X^h$. In the applications in Section~\ref{sec:applications} we ensure that $\beta$ exists by numerically computing an approximation $\widetilde \beta$ with the approach presented in \cite{ballarin2015supremizer}, which we briefly summarize here. Let us suppose that $X_\mathcal{X} \in \mathbb{R}^{(n^{(1)} + n^{(2)}) \times (n^{(1)} + n^{(2)})}$ and $X_\Lambda \in \mathbb{R}^{n_\Gamma \times n_\Gamma}$ are norm matrices such that $\Vert v^h \Vert^2_\mathcal{Y} = (X_\mathcal{Y} \textbf v, \textbf v)$ and $\Vert \eta^\delta \Vert^2_\Lambda = (X_\Lambda \boldsymbol \eta, \boldsymbol \eta)$ for every $v \in \mathcal X^h$ and every $\eta \in \Lambda^\delta$. In the previous expressions, we denoted $(\cdot,\cdot)$ the standard scalar product in $\mathbb{R}^m$ ($m = n^{(1)}+n^{(2)}$ or $m = n_\Gamma$) and $\textbf v$ and $\boldsymbol \eta$ the vectors of degrees of freedom of $v^h$ and $\eta^\delta$. Then, we have
\begin{align}
\widetilde \beta &= \inf_{\eta^\delta \in \Lambda^\delta} \sup_{v^h \in \mathcal X^h} \dfrac{b(v^h,\eta^\delta)}{\Vert v^h \Vert_{\mathcal{X}}\Vert \eta^\delta \Vert_{\Lambda}} = \inf_{\boldsymbol \eta \neq \textbf 0} \sup_{\textbf v \neq \textbf 0} \dfrac{\left ( B \textbf v, \boldsymbol \eta \right )}{(X_\mathcal{X} \textbf v, \textbf v)^{1/2} (X_\Lambda \boldsymbol \eta, \boldsymbol \eta)^{1/2}} \\ &= \inf_{\boldsymbol \eta \neq \textbf 0} \dfrac{1}{(X_\Lambda \boldsymbol \eta, \boldsymbol \eta)^{1/2}} \sup_{\textbf w = X_\mathcal{X}^{1/2} \textbf v \neq \textbf 0}\dfrac{\left ( \textbf w, X_\mathcal{X}^{-1/2} B^T \boldsymbol \eta \right )}{(\textbf w, \textbf w)^{1/2}} \\ &= \inf_{\boldsymbol \eta \neq \textbf 0} \dfrac{\left ( X_\mathcal{X}^{-1/2} B^T \boldsymbol \eta, X_\mathcal{X}^{-1/2} B^T \boldsymbol \eta\right )^{1/2} }{(X_\Lambda \boldsymbol \eta, \boldsymbol \eta)^{1/2}} = \inf_{\boldsymbol \eta \neq \textbf 0} \dfrac{\left ( B X_\mathcal{X}^{-1} B^T \boldsymbol \eta, \boldsymbol \eta\right )^{1/2} }{(X_\Lambda \boldsymbol \eta, \boldsymbol \eta)^{1/2}}.
\end{align}
Introducing now the following generalized eigenvalue problem
\begin{equation}
\begin{bmatrix}
X_\mathcal{X} & B^T \\
B & 0
\end{bmatrix}
\begin{bmatrix}
\textbf v \\
\boldsymbol \eta
\end{bmatrix}
=
- \sigma
\begin{bmatrix}
0 & 0 \\
0 & X_\Lambda
\end{bmatrix}
\begin{bmatrix}
\textbf v \\
\boldsymbol \eta
\end{bmatrix},
\label{eq:generalized}
\end{equation}
and recognizing that we have
\begin{equation}
B X^{-1}_\mathcal{X} B^T \boldsymbol \eta = \sigma X_\Lambda \boldsymbol \eta \quad \Rightarrow \quad  \sigma = \dfrac{\left ( B X^{-1}_\mathcal{X} B^T \boldsymbol \eta, \boldsymbol \eta \right ) }{\left ( X_\Lambda \boldsymbol \eta, \boldsymbol \eta \right ) },
\label{eq:generalized_2}
\end{equation}
we conclude that $\widetilde \beta$ can be computed as the square root of the minimum eigenvalue of Eq.~\eqref{eq:generalized}, i.e. $\widetilde \beta = \sqrt{\sigma_{\text{min}}}$. For an application of this strategy, we refer the reader to the results presented in Fig.~\ref{fig:infsup}.

\subsection{Convergence result for saddle-point problems}
We close this Section by focusing on the convergence of problem \weakref{weak:saddleweak_h}. The following theorem gives a sharper bound than Eq.~\eqref{eq:estimate} to the estimate of the approximation error.
\begin{theorem}{\cite[Theorem 16.6]{quarteroni2014numerical}}
Let the assumptions of Theorem~\ref{theorem:existence} be satisfied. Then the solution $(u,\lambda)$ of \weakref{weak:saddleweak} and the solution $(u^h,\lambda^\delta)$ of \weakref{weak:saddleweak_h} satisfy the following error estimates
\begin{align}
&\Vert u - u^h \Vert_{\mathcal X} \leq \left (1 + \dfrac{\kappa_a}{\alpha} \right) \inf_{v^h_* \in \mathcal V^{h,\delta}} \Vert u - v_*^h \Vert_{\mathcal X} + \dfrac{\kappa_b}{\alpha}\inf_{\eta^\delta \in \Lambda^\delta} \Vert \lambda - \eta^\delta \Vert_{\Lambda},\label{eq:error_estimate}\\
&\Vert \lambda - \lambda^\delta \Vert_{\Lambda} \leq \dfrac{\kappa_a}{\beta}\left (1+\dfrac{\kappa_a}{\alpha} \right ) \inf_{v_*^h \in \mathcal V^{h,\delta}} \Vert u - v^h_* \Vert_{\mathcal X} + \left (1 + \dfrac{\kappa_b}{\beta} + \dfrac{\kappa_a \kappa_b}{\alpha \beta} \right ) \inf_{\eta^\delta \in \Lambda^\delta} \Vert \lambda - \eta^\delta \Vert_{\Lambda}.
\end{align}
Moreover, the following error estimate holds
\begin{equation}
\inf_{v_*^h \in \mathcal V^{h,\delta}} \Vert u - v_*^h \Vert_{\mathcal X} \leq \left ( 1 + \dfrac{\kappa_b}{\beta} \right ) \inf_{v^h \in \mathcal X^h} \Vert u - v^h \Vert_{\mathcal X}.
\label{eq:best_appr}
\end{equation}
\label{theorem:stability}
\end{theorem}
Theorem \ref{theorem:stability} shows that, whenever the space of Lagrange multipliers is rich enough (namely the second term in Eq.~\eqref{eq:error_estimate} becomes negligible compared the first one), the approximation of $u$ is essentially bounded by the best approximation error on $\mathcal X$. However, this richness may lower the inf-sup constant $\beta$ and therefore loose the approximation \eqref{eq:best_appr}. It is therefore important to find the correct balance.

We remark that in Eq.~\eqref{eq:strang} we have that increasing the size of the Lagrange multipliers space is equivalent to lowering the size of $V^{h,\delta}$ and, consequently, the supremum in its right hand side. The two error estimates in \eqref{eq:strang} and \eqref{eq:error_estimate} are therefore two equivalent ways of expressing the fact that, if the continuity over the interface $\Gamma$ is enforced strongly enough, the error converges to zero as the error due to the spatial discretization in $\mathcal X^h$. As we show in the next section, we are then able to recover the usual convergence orders for $u$ with respect to the mesh size $h$ when using the finite element method.

\label{subsec:infsup}
\section{Numerical results}
\label{sec:applications}
\newlength
\figureheight
\newlength
\figurewidth

In this section,  we focus on the performance of the method presented in Section~\ref{sec:method} on two-dimensional problems defined over the unit square. The numerical simulations we present are performed with a set of  \textsc{Matlab} scripts which can be freely downloaded\footnote{\url{https://github.com/lucapegolotti/coupling_scripts}}.

For all the simulations, we employ standard piecewise polynomial Lagrangian basis functions defined over suitable triangulations in the subdomains. Regarding the choice of basis functions for $\Lambda^\delta$, we already anticipated in Section~\ref{sec:method} that in this paper we investigate the possibility of using low-frequency Fourier basis functions built on the interface $\Gamma$.

\subsection{Choice of basis functions for the Lagrange multipliers}
Given an interface with length $L$, we consider $\xi_1 = 1$ and, for $i = 1,\ldots,n_\omega$
\begin{equation}
\xi_{2i}(s) =\sin(\omega_i \pi s), \quad \xi_{2i+1} (s) = \cos(\omega_i  \pi s),
\label{eq:fourier_b}
\end{equation}
where $s$ is the arc length of the interface $\Gamma$, $\omega_i = i / L$, and $n_\omega$ is the number of considered frequencies; it holds that $n_\Gamma = 2n_\omega + 1$. With this definition, the set $\{\xi_i\}_{i = 1}^{n_\Gamma}$ forms an orthogonal basis with respect to the $L^2(0,2L)$ scalar product. We choose to employ such basis -- instead of the standard Fourier basis orthogonal (or orthonormal) with respect to the $L^2(0,L)$ scalar product -- because, by considering basis functions with periodicity $L$, we would impose an unnecessary periodicity constraint, in particular, the equality of the functions in $\Lambda^\delta$ and their derivatives at the extrema of $\Gamma$. As a result, we empirically observed that by employing the standard $L^2(0,L)$ orthonormal Fourier basis functions the optimal convergence of the finite element method is retrieved for larger values of $n_\Gamma$ compared to the choice in Eq.~\eqref{eq:fourier_b}. However, utilizing non-orthonormal basis functions \eqref{eq:fourier_b} has a dramatic influence on the condition number of the resulting linear system, which has exponential growth with the increasing number of basis functions on the interface; see Fig.~\ref{fig:condition} (left).

\begin{figure}
\centering
\setlength
\figureheight{0.3\textwidth}
\setlength
\figurewidth{0.8\textwidth}
%
%
\definecolor{mycolor1}{rgb}{0.00000,0.44700,0.74100}%
\definecolor{mycolor2}{rgb}{0.85000,0.32500,0.09800}%
\definecolor{mycolor3}{rgb}{0.92900,0.69400,0.12500}%
\definecolor{mycolor4}{rgb}{0.49400,0.18400,0.55600}%
\definecolor{mycolor5}{rgb}{0.46600,0.67400,0.18800}%
\begin{tikzpicture}

\begin{axis}[%
width=0.4\figurewidth,
height=0.3\figurewidth,
at={(0\figurewidth,0\figureheight)},
scale only axis,
xmin=1.0000000000,
xmax=31.0000000000,
xlabel style={font=\color{white!15!black}},
xlabel={$n_\Gamma$},
ymode=log,
ymin=1000.0000000000,
ymax=100000000000000000000.0000000000,
xminorticks=true,
xmajorgrids,
xminorgrids,
title={Non-orthonormal basis functions},
yminorticks=true,
ymajorgrids,
yminorgrids,
colorbrewer cycle list=Paired,
ylabel style={font=\color{white!15!black}},
ylabel={condition number},
axis background/.style={fill=white},
legend style={at={(1.03,1)},anchor=north west,legend cell align=left,align=left,draw=white!15!black},
yminorgrids=false,xminorgrids=false,xminorticks=false,yminorticks=false,grid style={very thin,gray!13},xtick={1,7,13,19,25,31}
]
\addplot 
  table[row sep=crcr]{%
1.0000000000	3689.6734970349\\
3.0000000000	8466.9166442601\\
5.0000000000	81278.7551884233\\
7.0000000000	2655022.3043047558\\
9.0000000000	109642641.2565604895\\
11.0000000000	4509082710.3547477722\\
13.0000000000	189207435911.7423400879\\
15.0000000000	9293347079329.8632812500\\
17.0000000000	575445292643094.6250000000\\
19.0000000000	41814821916918568.0000000000\\
21.0000000000	841167761742020992.0000000000\\
23.0000000000	266777500399624064.0000000000\\
25.0000000000	1010937137651919360.0000000000\\
27.0000000000	6162807030888860672.0000000000\\
29.0000000000	9387001889389410304.0000000000\\
31.0000000000	45615766028489711616.0000000000\\
};

\addplot 
  table[row sep=crcr]{%
1.0000000000	7090.1825416079\\
3.0000000000	14984.3476592650\\
5.0000000000	99803.8567893715\\
7.0000000000	2590383.5507376948\\
9.0000000000	100815703.2606064081\\
11.0000000000	4039291582.4084076881\\
13.0000000000	159303259613.1100463867\\
15.0000000000	6316119766933.3974609375\\
17.0000000000	268426781905627.7812500000\\
19.0000000000	13003869027502642.0000000000\\
21.0000000000	513728865722882688.0000000000\\
23.0000000000	1663168090210963200.0000000000\\
25.0000000000	25157175581648523264.0000000000\\
27.0000000000	9650211531001403392.0000000000\\
29.0000000000	12913308195162056704.0000000000\\
31.0000000000	33284309796770238464.0000000000\\
};

\addplot 
  table[row sep=crcr]{%
1.0000000000	14262.7996490092\\
3.0000000000	28667.9309565070\\
5.0000000000	140546.8126874399\\
7.0000000000	2665815.1673817136\\
9.0000000000	95884635.4824651033\\
11.0000000000	3735046632.4657125473\\
13.0000000000	144702666365.3332824707\\
15.0000000000	5530182159431.6201171875\\
17.0000000000	210193042624900.1875000000\\
19.0000000000	8299307144997054.0000000000\\
21.0000000000	346210667448869312.0000000000\\
23.0000000000	513437307213550144.0000000000\\
25.0000000000	1246982959321596928.0000000000\\
27.0000000000	1869426852974389760.0000000000\\
29.0000000000	8222254279673583616.0000000000\\
31.0000000000	6002601629170775040.0000000000\\
};

\addplot +[mark options={solid}]
  table[row sep=crcr]{%
1.0000000000	27693.8269522933\\
3.0000000000	54198.6525930242\\
5.0000000000	217189.2487661901\\
7.0000000000	2911833.5830472191\\
9.0000000000	94187995.5316184610\\
11.0000000000	3568838122.6861977577\\
13.0000000000	136123179074.5238494873\\
15.0000000000	5143595406258.0576171875\\
17.0000000000	192340322169935.4687500000\\
19.0000000000	7113712625634178.0000000000\\
21.0000000000	786099251103312768.0000000000\\
23.0000000000	230509051515236096.0000000000\\
25.0000000000	1068147662895629568.0000000000\\
27.0000000000	967967955555135744.0000000000\\
29.0000000000	2739448825642392064.0000000000\\
31.0000000000	29129830690413338624.0000000000\\
};

\addplot 
  table[row sep=crcr]{%
1.0000000000	56128.0160683956\\
3.0000000000	108117.2004595782\\
5.0000000000	379088.5515663907\\
7.0000000000	3486471.0461534043\\
9.0000000000	95003993.0293032825\\
11.0000000000	3477798113.4906415939\\
13.0000000000	130698987157.5382232666\\
15.0000000000	4883501744291.1845703125\\
17.0000000000	180988435082132.8125000000\\
19.0000000000	6722721113606098.0000000000\\
21.0000000000	172044774249227456.0000000000\\
23.0000000000	2821137698946969088.0000000000\\
25.0000000000	1554960917187383040.0000000000\\
27.0000000000	463317788148246144.0000000000\\
29.0000000000	66474870347964186624.0000000000\\
31.0000000000	1113647525663338112.0000000000\\
};
\end{axis}

\begin{axis}[%
width=0.4\figurewidth,
height=0.3\figurewidth,
at={(0.5\figurewidth,0\figureheight)},
scale only axis,
xmin=1.0000000000,
xmax=31.0000000000,
xlabel style={font=\color{white!15!black}},
xlabel={$n_\Gamma$},
ymode=log,
ymin=1000.0000000000,
ymax= 500000000000.0,
title={Orthonormal basis functions},
xminorticks=true,
xmajorgrids,
xminorgrids,
yminorticks=true,
ymajorgrids,
yminorgrids,
colorbrewer cycle list=Paired,
ylabel style={font=\color{white!15!black}},
axis background/.style={fill=white},
legend style={at={(1.03,1)},anchor=north west,legend cell align=left,align=left,draw=white!15!black},
yminorgrids=false,xminorgrids=false,xminorticks=false,yminorticks=false,grid style={very thin,gray!13},xtick={1,7,13,19,25,31}
]

\addplot 
  table[row sep=crcr]{%
1.0000000000	3689.6734970349\\
3.0000000000	3442.6161349874\\
5.0000000000	3435.6146890641\\
7.0000000000	3485.8036748755\\
9.0000000000	3572.9990663499\\
11.0000000000	3702.5655612852\\
13.0000000000	3875.1557817942\\
15.0000000000	4164.5562741114\\
17.0000000000	10458.5187150225\\
19.0000000000	38768.9101564362\\
21.0000000000	137094.1871607430\\
23.0000000000	870077.6543068386\\
25.0000000000	5784897.0889125848\\
27.0000000000	70234145.2464837432\\
29.0000000000	1163623607.5159008503\\
31.0000000000	28658194652.9319000244\\
};
\addlegendentry{$h = 1/20$}

\addplot 
  table[row sep=crcr]{%
1.0000000000	7090.1825416079\\
3.0000000000	6575.9208195863\\
5.0000000000	6520.3736351195\\
7.0000000000	6554.5735662290\\
9.0000000000	6621.5446649959\\
11.0000000000	6721.1111487348\\
13.0000000000	6858.6701921750\\
15.0000000000	7034.7880841457\\
17.0000000000	7265.7898224120\\
19.0000000000	7696.8552911952\\
21.0000000000	18603.9634699546\\
23.0000000000	56017.1807327024\\
25.0000000000	163961.0369104618\\
27.0000000000	586348.6679602657\\
29.0000000000	3259349.3954603253\\
31.0000000000	16351737.3424132708\\
};
\addlegendentry{$h = 1/28$}

\addplot 
  table[row sep=crcr]{%
1.0000000000	14262.7996490092\\
3.0000000000	13180.8857163718\\
5.0000000000	13025.0202167756\\
7.0000000000	13036.5894169097\\
9.0000000000	13088.0445637238\\
11.0000000000	13164.6068950857\\
13.0000000000	13269.5268083299\\
15.0000000000	13407.8306768638\\
17.0000000000	13580.9357779188\\
19.0000000000	13791.4036176265\\
21.0000000000	14066.2435450175\\
23.0000000000	14515.6015025347\\
25.0000000000	21312.2823644256\\
27.0000000000	51990.8183627827\\
29.0000000000	131165.8879543557\\
31.0000000000	322540.8873021581\\
};
\addlegendentry{$h = 1/40$}

\addplot
  table[row sep=crcr]{%
1.0000000000	27693.8269522933\\
3.0000000000	25543.4015467135\\
5.0000000000	25199.9719436426\\
7.0000000000	25174.4039041313\\
9.0000000000	25211.2931967059\\
11.0000000000	25273.1138193935\\
13.0000000000	25356.1309062111\\
15.0000000000	25463.7892027951\\
17.0000000000	25600.5678420501\\
19.0000000000	25768.8739088760\\
21.0000000000	25969.0357468576\\
23.0000000000	26205.9901535484\\
25.0000000000	26505.4118535947\\
27.0000000000	26947.6410826392\\
29.0000000000	27757.2116602747\\
31.0000000000	43468.9437488515\\
};
\addlegendentry{$h = 1/56$}

\addplot
  table[row sep=crcr]{%
1.0000000000	56128.0160683957\\
3.0000000000	51706.9577147704\\
5.0000000000	50965.5687039411\\
7.0000000000	50863.9977122355\\
9.0000000000	50878.1761811366\\
11.0000000000	50926.5129021124\\
13.0000000000	50994.0829697185\\
15.0000000000	51079.5241697314\\
17.0000000000	51185.4543563998\\
19.0000000000	51315.6076997594\\
21.0000000000	51473.1048466165\\
23.0000000000	51659.3354098637\\
25.0000000000	51874.5697865772\\
27.0000000000	52121.7360688892\\
29.0000000000	52414.1808329473\\
31.0000000000	52789.5387668249\\
};
\addlegendentry{$h = 1/80$}

\end{axis}
\end{tikzpicture}%
\caption{Condition number of the discretized matrix in Eq.~\eqref{eq:algebraic} for the Poisson problem on two subdomains (see Section~\ref{subsec:poisson}) with conforming meshes vs number of basis functions for the Lagrange multiplier space, with different refinement levels of the (conforming) meshes. On the left, we consider non-orthonormal ``half"  Fourier basis functions \eqref{eq:fourier_b}, while on the right we consider their orthonormalization.}
\label{fig:condition}
\end{figure}
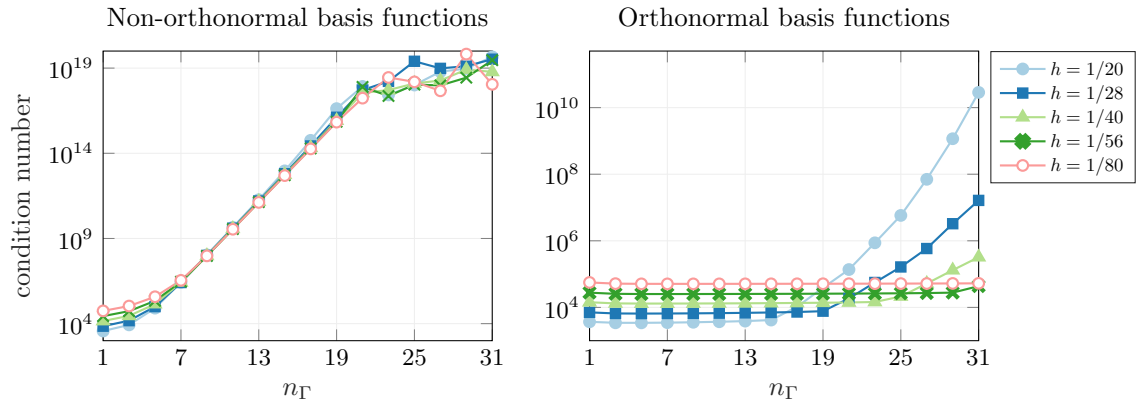

In order to retain the convergence order attained by using the Fourier modes in Eq.~\eqref{eq:fourier_b} and, at the same time, control the condition number of the system, we propose an orthonormalization strategy based on the Gram-Schmidt algorithm or, equivalently, on the QR decomposition \cite{ruhe1983numerical}. Even though the coefficients of the orthonormal basis generated by \eqref{eq:fourier_b} with these algorithms could be analytically derived, their exact expression quickly becomes complex with $n_\Gamma$ becoming large. With our approach, we aim at obtaining an approximation of such coefficients relying on a fine sampling of the basis functions on the interval $(0,L)$. We remark that, in addition to allowing to effortlessly compute a large number of orthonormal basis functions, our approach has the advantage to be general enough to be applied to any set of non-orthonormal basis functions.

Let $\{\xi_i\}_{i = 1}^{n_\Gamma}$ be the set of non-orthonormal basis functions defined on $\Gamma$. Moreover, let $\{x_i\}_{i = 1}^{n_\text{s}}$ be distinct sample points distributed over the interval $(0,L)$, where $L$ still denotes the length of the interface. We now introduce the functions $\{\kappa_i\}_{i = 1}^{n_\text{s}}$, which we identify with the set of standard Lagrangian piecewise linear basis functions centered on each sample point $x_i$, and the associated mass matrix $M_{ij} = \int_{\Gamma} \kappa_i \kappa_j\,\text{d}\textbf x$. Let $V = [\mathbf v_1,\, \mathbf v_2,\,\ldots, \mathbf v_{n_\Gamma}] \in \mathbb{R}^{n_{\text{s}} \times n_\Gamma}$ be the matrix of the evaluations of the basis functions on the sample points, namely $V_{ij} = \xi_j(x_i)$. We remark that, for each $i = 1,2,\ldots,n_\Gamma$, we have $\Vert \xi_i \Vert_{L^2(0,L)}^2 \approx \textbf v_i^T M \textbf v_i$. Since $M$ is a positive-definite matrix, it admits a unique Cholesky decomposition and there exists $C \in \mathbb{R}^{n_\text{s} \times n_\text{s}}$ such that $C^T C = M$.
Let us now consider the unit matrix $Q \in \mathbb{R}^{n_\text{s} \times n_\Gamma}$ and the upper triangular matrix $R \in \mathbb{R}^{n_\Gamma \times n_\Gamma}$ such that the truncated QR decomposition of $CV$ reads
\begin{equation}
C V = Q R.
\end{equation}
By construction, we have
\begin{equation}
(C^{-1} Q)^T M C^{-1} Q = Q^T C^{-T} C^{T} C C^{-1} Q = I,
\end{equation}
thus, the columns of $C^{-1}Q$ represent evaluations at the sample points of functions orthonormal on $(0,L)$ with respect to the $L^2$ product. The matrix $R$ performs the change of variable from the frame of reference of the new orthonormal basis functions to the frame of reference of the non-orthonormal basis functions. If the sampling is sufficiently fine, we speculate that the elements of the matrix $R^{-1}$ well approximate the coefficients which are computed by applying the Gram-Schmidt algorithm to the continuous non-orthonormal basis functions $\{\xi_i\}_{i = 1}^{n_\Gamma}$ and, in particular, that
\begin{equation}
\xi_i^\text{GS} = \sum_{j = 1}^{n_\Gamma} \xi_j R^{-1}_{ji}.
\label{eq:fourier_gs}
\end{equation}

\begin{figure}
\centering
\setlength
\figureheight{0.65\textwidth}
\setlength
\figurewidth{0.8\textwidth}
\input{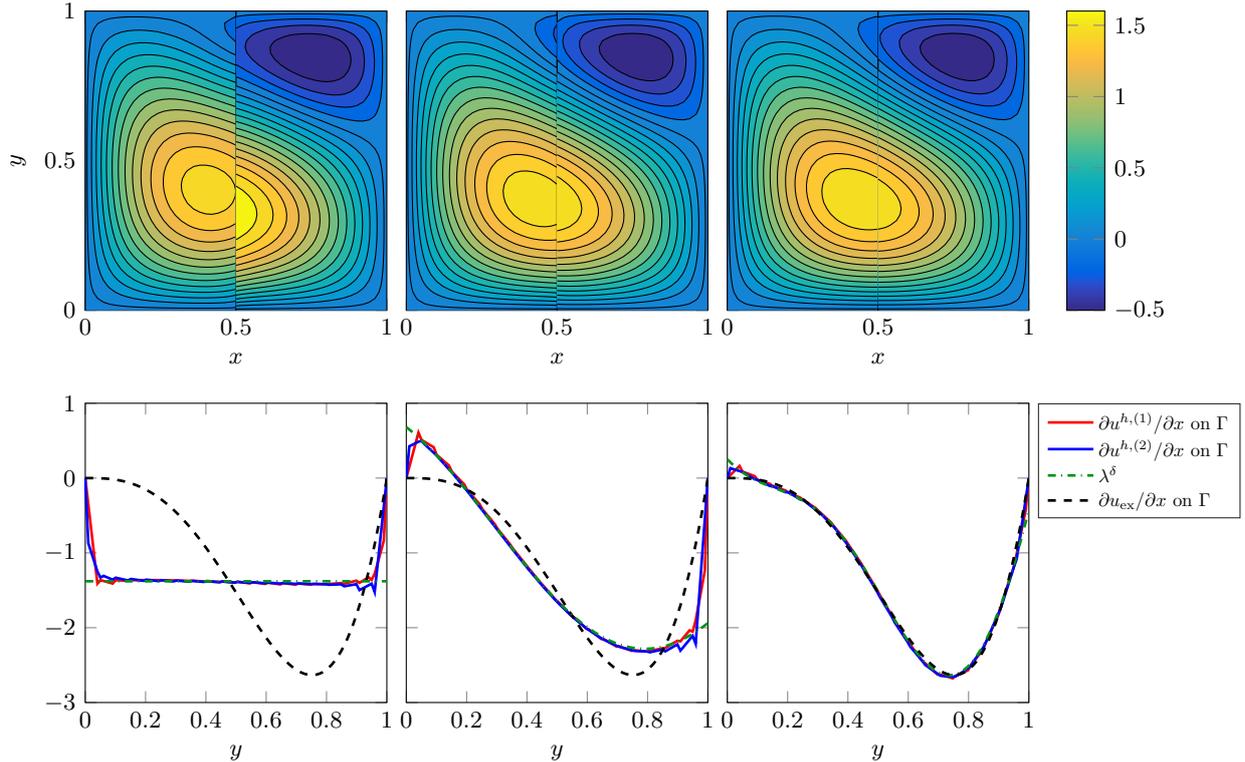}
\caption{Contour lines of the solution (top row) and derivative of the solution at the interface $\Gamma$ (bottom row) when $n_\Gamma = 1$ (left column), $n_\Gamma = 3$ Fourier modes (middle column) and $n_\Gamma = 5$ Fourier modes (right column) are used to characterize the space $\Lambda^\delta$. The red and blue solid lines in the bottom plots represent the partial derivatives with respect to $x$ -- computed at the interface $\Gamma$ -- of the numerical solutions in $\Omega_1$ and $\Omega_2$ respectively. The dash-dot green line is computed by reconstructing $\lambda^\delta$ as a linear combination of the Fourier basis functions, i.e. $\lambda^\delta = \sum_{i = 1}^{n_\Gamma} \lambda_{i} \xi_{i}$. The results are obtained on a mesh conforming at the interface, with quadratic Lagrangian polynomials on both domains and with mesh size $h = 1/20$.}
\label{fig:convergence_qual}
\end{figure}

From a practical perspective, the matrix $R^{-1}$ is suitable to compute the coupling matrix $B^{\text{GS}}$ with respect to the orthonormal Fourier basis functions, knowing the coupling matrix computed without orthonormalization $B$. Indeed, we have
\begin{equation}
 B^{\text{GS}}_{ij} = \int_{\Gamma} \xi_i^{\text{GS}} \varphi_j\,\text{d}\textbf s = \sum_{k = 1}^{n_\Gamma} \left( \int_{\Gamma} \xi_k \varphi_j\,\text{d}\textbf s \right) R^{-1}_{ki} = \sum_{k = 1}^{n_\Gamma} B_{kj} R^{-1}_{ki},
\end{equation}
or equivalently $B^{\text{GS}}_{ij} = R^{-T} B$. Therefore, the condition number of the system can be controlled by multiplying the coupling matrices $B$ by the matrix $R^{-1}$; observe that, being $R$ an upper triangular matrix, the application of its inverse is performed with negligible cost. The matrix $R^{-1}$ depends only on the choice of the non-orthonormal basis functions and can be then computed a priori. We remark that, with this approach, the orthonormal basis functions are never explicitly computed. Moreover, since the discrete space is exactly the same, the approximation properties and the convergence orders are not changed. Fig.~\ref{fig:condition} (right) shows that, after the orthonormalization of the Fourier basis functions \eqref{eq:fourier_b} by the algorithm we presented, the system is more stable and the condition number increases with the number of Fourier basis functions $n_\Gamma$ dependently on the refinement level of the mesh $h$.

\subsection{The Poisson problem}
\label{subsec:poisson}

Let us consider the global Poisson problem \eqref{eq:poisson} on the domain $\Omega = (0,1)^2$, where we take $f$ such that $u_{\text{ex}} = 100 x y(1-x)(1-y)\sin(1/3-xy^2)$ is the exact solution. We divide $\Omega$ into $\Omega_1 = (0,0.5) \times (0, 1)$ and $\Omega_2 = (0.5,1) \times (0, 1)$.

We numerically solve the problem on $\Omega_1$ and $\Omega_2$ by employing structured triangular conforming and non-conforming meshes with varying mesh size $h$. The conforming meshes are obtained by subdividing the domain in the $x$- and $y$-direction in the same number of elements. On the other hand, the non-conforming meshes are built by taking in the $y$-direction of $\Omega_2$ $N+1$ elements, $N$ being the number of elements in the $y$-direction in $\Omega_1$ as well as the total number of elements in the $x$-direction.

Fig.~\ref{fig:convergence_qual} shows how the solutions on $\Omega_1$ and $\Omega_2$ obtained with a conforming mesh with $N = 20$ elements in each direction change with respect to the number of basis functions on the interface. The results are obtained with quadratic Lagrangian polynomials in both subdomains. From the contour lines plots in the top row, it appears that the two solutions match quite accurately at the interface with 5 Fourier basis functions ($n_\omega = 2$). In the second row of Fig.~\ref{fig:convergence_qual}, we plot the approximation by finite differences of the derivative of the solution with respect to $x$ in the two domains, which is equal to the normal derivative of $u^{h,(1)}$  and to the opposite of the normal derivative on $u^{h,(2)}$ on $\Gamma$ respectively. Observe that, as we already highlighted in Remark~\ref{remark:derivative}, the Lagrange multiplier $\lambda^\delta$ takes the role of the normal derivative of $u^h$ on $\Gamma$.

Let us address the convergence of the global solution to the exact one with respect both to the mesh size $h$ and the number of basis functions on the interface $n_\Gamma$. To this end, we consider meshes with total number of elements in the $x$-direction $N = 20,28,40,56,80,114,160$ and we solve the problems by employing quadratic Lagrangian basis functions in both subdomains. Fig.~\ref{fig:convergence_quant} (top row) depicts the decaying of the error in $\mathcal X$-norm (the broken norm) with respect to $h$, as well as the convergence of the error obtained by solving the problem on a single mesh of $\Omega$ (in black dashed line). When employing both conforming and non-conforming meshes, the error is optimal -- in the sense that we recover the theoretical order of convergence $h^2$ of quadratic finite elements for the $H^1$-error -- when $n_\Gamma$ is large enough, e.g. $n_\Gamma \geq 13$. If $n_\Gamma$ is too small, on the contrary, the solution is unable to converge to the exact solution with $h$ and reaches a stagnation point. We remark that this result is perfectly consistent with Strang's second lemma \eqref{eq:strang} and with the stability result in Theorem \ref{theorem:stability}: whenever the space of Lagrange multiplier is rich enough (which is equivalent to requiring that $\mathcal V^{h,\delta}$ be a good approximation of $\mathcal V$), the best approximation error of the interpolation is recovered.
\begin{figure}
\centering
\setlength
\figureheight{0.7\textwidth}
\setlength
\figurewidth{0.8\textwidth}
\input{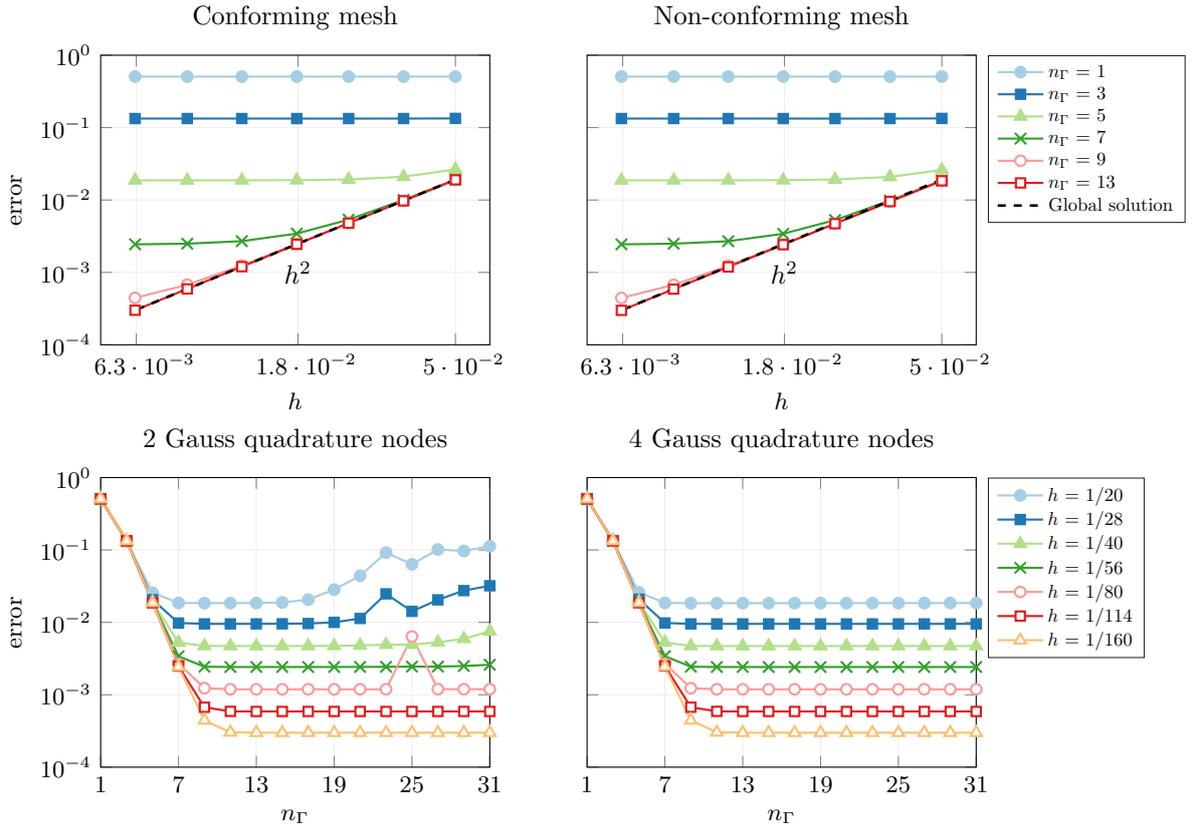}
\caption{The top row shows the convergence of global error with respect to the mesh size $h$ and number of basis functions on the interface $n_\Gamma$ with conforming and non-conforming meshes. The colored lines show the errors in the $\mathcal X$-norm, i.e. the broken norm. The black dashed line represents the $H^1$-error obtained by solving the global problem on a conforming mesh, which can be regarded as the union of two conforming meshes on $\Omega_1$ and $\Omega_2$. The bottom row shows the convergence of the error on the non-conforming meshes with respect to $n_\Gamma$ with varying mesh sizes; in particular, the discretization of the terms corresponding to the bilinear form $b(\cdot,\cdot)$ in Eq.~\eqref{eq:apprB} has been performed with 2 and 4 Gauss quadrature nodes on the left and on the right plots respectively. Note that the two plots at the right both refer to the same data obtained with a 4-node quadrature rule.}
\label{fig:convergence_quant}
\end{figure}
\begin{remark}
In our numerical simulations with non-conforming meshes, we observed that instabilities arise when using coarse meshes and low-order quadrature rules for the computations of the approximate integrals of $B_1$ and $B_2$ in Eq.~\eqref{eq:apprB}. Fig.~\ref{fig:convergence_quant} (bottom row, left) shows that, when using for example 2 Gauss quadrature nodes, the error increases with $n_\Gamma$ when $h = 1/20, 1/28,1/40$. By increasing the order of the quadrature rule and choosing 4 Gauss quadrature nodes this issue is completely fixed; see Fig.~\ref{fig:convergence_quant} (bottom row, right). The plots in Fig.~\ref{fig:convergence_quant} (right column) are obtained from the same data. We did not encounter stability problems when using conforming meshes, even with low-order quadrature rules.
\end{remark}

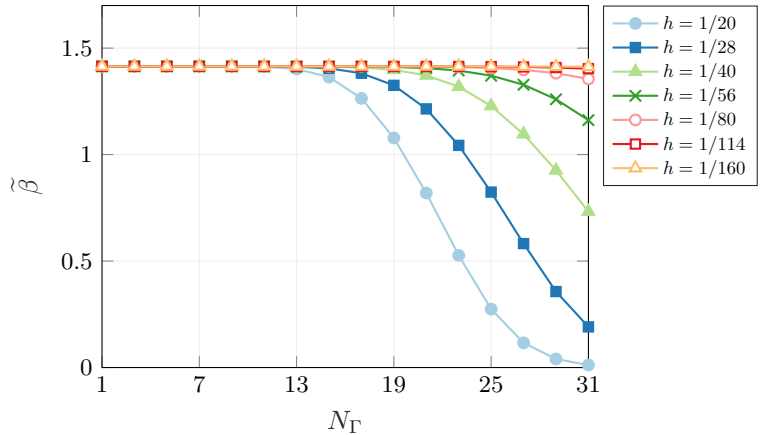
\begin{figure}
\centering
\setlength
\figureheight{0.4\textwidth}
\setlength
\figurewidth{0.4\textwidth}
%
%
\definecolor{mycolor1}{rgb}{0.00000,0.75000,0.75000}%
\definecolor{mycolor2}{rgb}{0.75000,0.00000,0.75000}%
\definecolor{mycolor3}{rgb}{0.75000,0.75000,0.00000}%
\begin{tikzpicture}

\begin{axis}[%
width=\figurewidth,
height=0.75\figurewidth,
at={(0\figurewidth,0\figureheight)},
scale only axis,
xmin=1.0000000000,
xmax=31.0000000000,
xlabel style={font=\color{white!15!black}},
xlabel={$N_\Gamma$},
ylabel={$\widetilde \beta$},
ymin=0.00,
ymax=1.7000000000,
xminorticks=true,
xmajorgrids,
xminorgrids,
yminorticks=true,
ymajorgrids,
yminorgrids,
colorbrewer cycle list=Paired,
ylabel style={font=\color{white!15!black}},
axis background/.style={fill=white},
legend style={at={(1.03,1)},anchor=north west,legend cell align=left,align=left,draw=white!15!black},
yminorgrids=false,xminorgrids=false,xminorticks=false,yminorticks=false,grid style={very thin,gray!13},xtick={1,7,13,19,25,31}
]
\addplot
  table[row sep=crcr]{%
1.0000000000	1.4142135624\\
3.0000000000	1.4142135545\\
5.0000000000	1.4142134453\\
7.0000000000	1.4142043526\\
9.0000000000	1.4140144069\\
11.0000000000	1.4122276159\\
13.0000000000	1.4021342340\\
15.0000000000	1.3640221469\\
17.0000000000	1.2640898515\\
19.0000000000	1.0782057793\\
21.0000000000	0.8194847817\\
23.0000000000	0.5267160222\\
25.0000000000	0.2743625832\\
27.0000000000	0.1159605435\\
29.0000000000	0.0405003572\\
31.0000000000	0.0116887371\\
};
\addlegendentry{$h = 1/20$}

\addplot
  table[row sep=crcr]{%
1.0000000000	1.4142135624\\
3.0000000000	1.4142135548\\
5.0000000000	1.4142134829\\
7.0000000000	1.4142122579\\
9.0000000000	1.4141839540\\
11.0000000000	1.4139050561\\
13.0000000000	1.4122197787\\
15.0000000000	1.4050384081\\
17.0000000000	1.3820490282\\
19.0000000000	1.3251551135\\
21.0000000000	1.2146427117\\
23.0000000000	1.0432775033\\
25.0000000000	0.8238446547\\
27.0000000000	0.5820079837\\
29.0000000000	0.3568611332\\
31.0000000000	0.1909038641\\
};
\addlegendentry{$h = 1/28$}

\addplot 
  table[row sep=crcr]{%
1.0000000000	1.4142135624\\
3.0000000000	1.4142135548\\
5.0000000000	1.4142134863\\
7.0000000000	1.4142131701\\
9.0000000000	1.4142096182\\
11.0000000000	1.4141721433\\
13.0000000000	1.4139349200\\
15.0000000000	1.4128565444\\
17.0000000000	1.4090391097\\
19.0000000000	1.3980581188\\
21.0000000000	1.3717991246\\
23.0000000000	1.3189259131\\
25.0000000000	1.2284802968\\
27.0000000000	1.0957269206\\
29.0000000000	0.9258707118\\
31.0000000000	0.7322027435\\
};
\addlegendentry{$h = 1/40$}

\addplot  +[mark options={solid}]
  table[row sep=crcr]{%
1.0000000000	1.4142135624\\
3.0000000000	1.4142135548\\
5.0000000000	1.4142134867\\
7.0000000000	1.4142132312\\
9.0000000000	1.4142123743\\
11.0000000000	1.4142069148\\
13.0000000000	1.4141710402\\
15.0000000000	1.4140015550\\
17.0000000000	1.4133714982\\
19.0000000000	1.4114280896\\
21.0000000000	1.4063021233\\
23.0000000000	1.3945294049\\
25.0000000000	1.3707243650\\
27.0000000000	1.3280383105\\
29.0000000000	1.2597807882\\
31.0000000000	1.1618898978\\
};
\addlegendentry{$h = 1/56$}

\addplot 
  table[row sep=crcr]{%
1.0000000000	1.4142135624\\
3.0000000000	1.4142135548\\
5.0000000000	1.4142134868\\
7.0000000000	1.4142132386\\
9.0000000000	1.4142126213\\
11.0000000000	1.4142111120\\
13.0000000000	1.4142059345\\
15.0000000000	1.4141830357\\
17.0000000000	1.4140949779\\
19.0000000000	1.4138124891\\
21.0000000000	1.4130275093\\
23.0000000000	1.4110935941\\
25.0000000000	1.4068065895\\
27.0000000000	1.3981716817\\
29.0000000000	1.3822631106\\
31.0000000000	1.3553269845\\
};
\addlegendentry{$h = 1/80$}

\addplot 
  table[row sep=crcr]{%
1.0000000000	1.4142135624\\
3.0000000000	1.4142135548\\
5.0000000000	1.4142134868\\
7.0000000000	1.4142132394\\
9.0000000000	1.4142126471\\
11.0000000000	1.4142114632\\
13.0000000000	1.4142091939\\
15.0000000000	1.4142043303\\
17.0000000000	1.4141912356\\
19.0000000000	1.4141513746\\
21.0000000000	1.4140379058\\
23.0000000000	1.4137485152\\
25.0000000000	1.4130779128\\
27.0000000000	1.4116477552\\
29.0000000000	1.4088159059\\
31.0000000000	1.4035761984\\
};
\addlegendentry{$h = 1/114$}

\addplot 
  table[row sep=crcr]{%
1.0000000000	1.4142135624\\
3.0000000000	1.4142135548\\
5.0000000000	1.4142134868\\
7.0000000000	1.4142132395\\
9.0000000000	1.4142126502\\
11.0000000000	1.4142115024\\
13.0000000000	1.4142095061\\
15.0000000000	1.4142062111\\
17.0000000000	1.4142007246\\
19.0000000000	1.4141908436\\
21.0000000000	1.4141705652\\
23.0000000000	1.4141242837\\
25.0000000000	1.4140175046\\
27.0000000000	1.4137840639\\
29.0000000000	1.4133053467\\
31.0000000000	1.4123763454\\
};
\addlegendentry{$h = 1/160$}

\end{axis}
\end{tikzpicture}%
\caption{Decaying of the inf-sup constant $\widetilde \beta$ with respect to $n_\Gamma$ computed on the Poisson problem on conforming meshes, using quadratic polynomial basis functions on both subdomains and the orthonormal basis functions $\xi_i^{\text{GS}}$ on the interface. The constant is approximated as the square root of the minimum eigenvalue of Eq.~\eqref{eq:generalized_2}.}
\label{fig:infsup}
\end{figure}

Fig.~\ref{fig:infsup} shows the variation of the estimate of the inf-sup constant $\widetilde \beta$ -- computed as the square root of the minimum eigenvalue of the generalized eigenvalue problem \eqref{eq:generalized}, as described in Section~\ref{subsec:infsup} -- when the number of basis functions on the interface changes; the estimate refers to the simulation of the Poisson equations with conforming meshes and quadratic polynomial basis functions. Due to the difficulties in computing the $H^{-1/2}_{00}$- norm for the Lagrange multiplier, we replaced the estimate given by Eq.~\eqref{eq:generalized} with a surrogate where the space $\mathcal X^h$ is substituted with the space spanned by the traces on $\Gamma$ of the finite element basis functions $\varphi_i$; the $L^2$-norm is used both for such space and $\Lambda^\delta$. For the result in Fig.~\ref{fig:infsup}, we employed the orthonormal basis functions $\xi_i^{\text{GS}}$ computed as in Eq.~\eqref{eq:fourier_gs}, so that $X_\Lambda = I$; therefore, from Eq.~\eqref{eq:generalized_2} it follows that $\widetilde \beta$ is simply found as the square root of the minimum eigenvalue of $B^{\text{GS}} X_\mathcal{X}^{-1} (B^{\text{GS}})^T$. In Fig.~\ref{fig:infsup}, each curve presents a plateau phase in which the inf-sup constant stays approximately constant at around $\widetilde \beta \approx 1.41$ with the increment of $n_\Gamma$. The amplitude of such plateau phase increases when $h$ becomes smaller. Indeed, we observe that $\widetilde \beta$ starts decreasing for smaller values of $n_\Gamma$ when the meshes are coarser and that, conversely, for finer meshes the inf-sup constant varies relatively little in the range $n_\Gamma \in (1,31)$. We remark that, combined with the condition number shown in Fig.~\ref{fig:condition}~(right), this result ensures that for each refinement level, we are able to obtain the optimal convergence of the finite element method when the basis functions are orthonormal. Indeed, refining the mesh has the effect of both increasing the range of stability of the linear system -- see Fig.~\ref{fig:condition} (right) -- and increasing the number of basis functions at the interface that can be employed without reaching the fast decaying region of $\widetilde \beta$ in Fig.~\ref{fig:infsup}. With regard to this last point, we recall that it is important to prevent the inf-sup constant to become too small because it appears at the denominator of the constant multiplying the best approximation errors on $u$ and on $\lambda$ in the error estimates of Theorem~\ref{theorem:stability}.

We focus now on the solution of the problem when employing non-conforming meshes, linear Lagrangian basis functions in $\Omega_1$ and quadratic Lagrangian basis functions in $\Omega_2$. Fig.~\ref{fig:convergence_quant_p1p2} shows that the $H^1$-error scales in the two subdomains as the best approximation error of the local (to the subdomain) basis: we recover first order convergence in $\Omega_1$ and second order convergence in $\Omega_2$. We remark that the convergence of the global error in the broken norm is determined by the rate in $\Omega_1$ --  being the error in such subdomain much larger than that in $\Omega_2$ -- and it is of first order. In Fig.~\ref{fig:convergence_quant_p1p2} we also show with black dashed lines the global $H^1$ error obtained when solving the problem with linear (in the left plot) and quadratic (in the right plot) basis functions on the whole $\Omega$. As expected, the accuracy obtained with mixed polynomial degrees lies between the accuracies achieved while using only linear and only quadratic basis functions.

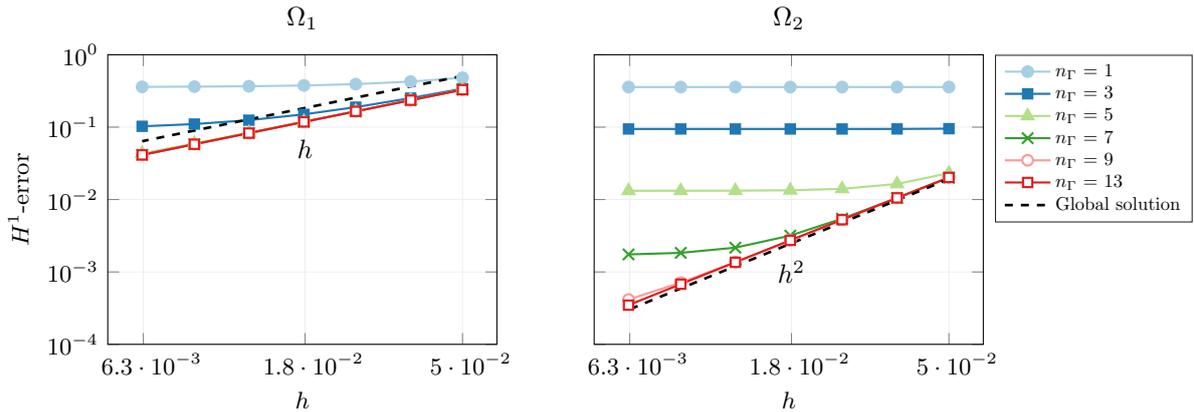
\begin{figure}
\centering
\setlength
\figureheight{0.7\textwidth}
\setlength
\figurewidth{0.8\textwidth}
%
%
\definecolor{mycolor1}{rgb}{0.00000,0.57650,0.09020}%
\definecolor{mycolor2}{rgb}{0.20000,0.60000,0.90000}%
\definecolor{mycolor3}{rgb}{0.75000,0.25000,0.80000}%
\definecolor{mycolor4}{rgb}{0.75000,0.75000,0.00000}%
\definecolor{mycolor5}{rgb}{0.90000,0.50000,0.90000}%
\definecolor{mycolor6}{rgb}{0.00000,0.57647,0.09020}%
\begin{tikzpicture}

\begin{axis}[%
title={$\Omega_1$},
width=0.4\figurewidth,
height=0.3\figurewidth,
at={(0\figurewidth,0\figureheight)},
scale only axis,
xmode=log,
xmin=0.0050000000,
xmax=0.0625000000,
xminorticks=true,
xlabel={$h$},
xmajorgrids,
xminorgrids,
ymode=log,
ymin=0.0001000000,
ymax=1.0000000000,
yminorticks=true,
ylabel={$H^1$-error},
ymajorgrids,
yminorgrids,
axis background/.style={fill=white},
colorbrewer cycle list=Paired,
yminorgrids=false,xminorgrids=false,xminorticks=false,yminorticks=false,grid style={very thin,gray!13},xtick={0.0063,0.018,0.05},xticklabel style={/pgf/number format/sci},xticklabel={ \pgfmathfloatparsenumber{\tick} \pgfmathfloatexp{\pgfmathresult} \pgfmathprintnumber{\pgfmathresult} }
]
\addplot
  table[row sep=crcr]{%
0.0500000000	0.4815022105\\
0.0357142857	0.4258477303\\
0.0250000000	0.3927621169\\
0.0178571429	0.3761153077\\
0.0125000000	0.3669565679\\
0.0087719298	0.3624118460\\
0.0062500000	0.3602259165\\
};

\addplot
  table[row sep=crcr]{%
0.0500000000	0.3394355442\\
0.0357142857	0.2518777389\\
0.0250000000	0.1888768494\\
0.0178571429	0.1501702817\\
0.0125000000	0.1247914400\\
0.0087719298	0.1103363337\\
0.0062500000	0.1027354644\\
};

\addplot 
  table[row sep=crcr]{%
0.0500000000	0.3288285632\\
0.0357142857	0.2358281587\\
0.0250000000	0.1656517814\\
0.0178571429	0.1187896100\\
0.0125000000	0.0837262456\\
0.0087719298	0.0595173528\\
0.0062500000	0.0434113479\\
};

\addplot +[mark options={solid}]
  table[row sep=crcr]{%
0.0500000000	0.3288834708\\
0.0357142857	0.2356633720\\
0.0250000000	0.1652581431\\
0.0178571429	0.1181483592\\
0.0125000000	0.0827514959\\
0.0087719298	0.0580976570\\
0.0062500000	0.0414168122\\
};

\addplot 
  table[row sep=crcr]{%
0.0500000000	0.3289487989\\
0.0357142857	0.2356910351\\
0.0250000000	0.1652667791\\
0.0178571429	0.1181460543\\
0.0125000000	0.0827397176\\
0.0087719298	0.0580761399\\
0.0062500000	0.0413840160\\
};

\addplot 
  table[row sep=crcr]{%
0.0500000000	0.3290171505\\
0.0357142857	0.2357156598\\
0.0250000000	0.1652744151\\
0.0178571429	0.1181485429\\
0.0125000000	0.0827404380\\
0.0087719298	0.0580761790\\
0.0062500000	0.0413836463\\
};

\addplot [color=black, dashed, line width=1.0pt]
  table[row sep=crcr]{%
0.0500000000	0.5093574257\\
0.0357142857	0.3646181571\\
0.0250000000	0.2555284555\\
0.0178571429	0.1826199853\\
0.0125000000	0.1278710966\\
0.0087719298	0.0897468091\\
0.0062500000	0.0639489276\\
};

\node at (axis cs: 0.018, 5e-2) {$h$};

\end{axis}

\begin{axis}[%
title={$\Omega_2$},
width=0.4\figurewidth,
height=0.3\figurewidth,
at={(0.5\figurewidth,0\figureheight)},
scale only axis,
xmode=log,
xmin=0.0050000000,
xmax=0.0625000000,
xminorticks=true,
xlabel={$h$},
xmajorgrids,
xminorgrids,
ymode=log,
yticklabels={,,},
ymin=0.000100,
ymax=1.0000000000,
yminorticks=true,
ymajorgrids,
yminorgrids,
colorbrewer cycle list=Paired,
axis background/.style={fill=white},
legend style={at={(1.03,1)},anchor=north west,legend cell align=left,align=left,draw=white!15!black},
yminorgrids=false,xminorgrids=false,xminorticks=false,yminorticks=false,grid style={very thin,gray!13},xtick={0.0063,0.018,0.05},xticklabel style={/pgf/number format/sci},xticklabel={ \pgfmathfloatparsenumber{\tick} \pgfmathfloatexp{\pgfmathresult} \pgfmathprintnumber{\pgfmathresult} }
]
\addplot 
  table[row sep=crcr]{%
0.0500000000	0.3581895599\\
0.0357142857	0.3580006348\\
0.0250000000	0.3579651847\\
0.0178571429	0.3579644284\\
0.0125000000	0.3579682405\\
0.0087719298	0.3579711734\\
0.0062500000	0.3579728255\\
};
\addlegendentry{$n_\Gamma$ = 1}

\addplot 
  table[row sep=crcr]{%
0.0500000000	0.0952497518\\
0.0357142857	0.0944294167\\
0.0250000000	0.0942788690\\
0.0178571429	0.0942785514\\
0.0125000000	0.0942969958\\
0.0087719298	0.0943107078\\
0.0062500000	0.0943183539\\
};
\addlegendentry{$n_\Gamma$ = 3}

\addplot 
  table[row sep=crcr]{%
0.0500000000	0.0231088206\\
0.0357142857	0.0164226693\\
0.0250000000	0.0140477849\\
0.0178571429	0.0134297615\\
0.0125000000	0.0132636352\\
0.0087719298	0.0132265714\\
0.0062500000	0.0132193384\\
};
\addlegendentry{$n_\Gamma$ = 5}

\addplot +[mark options={solid}]
  table[row sep=crcr]{%
0.0500000000	0.0199951811\\
0.0357142857	0.0105210097\\
0.0250000000	0.0054524882\\
0.0178571429	0.0031765192\\
0.0125000000	0.0021653554\\
0.0087719298	0.0018351607\\
0.0062500000	0.0017463284\\
};
\addlegendentry{$n_\Gamma$ = 7}

\addplot
  table[row sep=crcr]{%
0.0500000000	0.0201206083\\
0.0357142857	0.0105025379\\
0.0250000000	0.0052539917\\
0.0178571429	0.0027307944\\
0.0125000000	0.0013720219\\
0.0087719298	0.0007124956\\
0.0062500000	0.0004157747\\
};
\addlegendentry{$n_\Gamma$ = 9}

\addplot 
  table[row sep=crcr]{%
0.0500000000	0.0202014767\\
0.0357142857	0.0105436780\\
0.0250000000	0.0052700899\\
0.0178571429	0.0027322290\\
0.0125000000	0.0013589918\\
0.0087719298	0.0006781740\\
0.0062500000	0.0003481801\\
};
\addlegendentry{$n_\Gamma$ = 13}

\addplot [color=black,dashed,line width=1.0pt]
  table[row sep=crcr]{%
0.0500000000	0.0191686779\\
0.0357142857	0.0097978346\\
0.0250000000	0.0048057008\\
0.0178571429	0.0024530498\\
0.0125000000	0.0012023002\\
0.0087719298	0.0005921579\\
0.0062500000	0.0003006307\\
};
\addlegendentry{Global solution}

\node at (axis cs: 0.018, 1e-3) {$h^2$};

\end{axis}
\end{tikzpicture}%
\caption{Convergence of the $H^1$-error in $\Omega_1$ (left) and $\Omega_2$ (right) when linear and quadratic Lagrangian basis functions are employed in the two domains respectively. The meshes are not conforming at the interface. The black dashed lines refer to the global $H^1$-error obtained with linear (on the left) and quadratic (on the right) elements on the whole $\Omega$.}
\label{fig:convergence_quant_p1p2}
\end{figure}

\subsection{The Navier Stokes equations}
\label{sec:nstokes}

In this section, we test the flexibility of our method by solving the Navier-Stokes equations on $\Omega = (0,1) \times (0,1)$
\begin{equation}
\begin{alignedat}{3}
- \mu \Delta \mathbf u + (\mathbf u\cdot \nabla) \mathbf u + \nabla p &= \mathbf f \quad &&\text{in } \Omega, \\
\text{div} \mathbf u &= 0 \quad &&\text{in } \Omega, \\
\mathbf u &= \mathbf g &&\text{on } \Gamma_D,\\
\sigma(\mathbf u,p) \mathbf n &= \mathbf h &&\text{on } \Gamma_N,\\
\end{alignedat}
\label{eq:ns}
\end{equation}
where $\mathbf u$ and $p$ are velocity and pressure respectively, $\mu \in \mathbb{R}$ is the viscosity, $\Gamma_D$ and $\Gamma_N$ are portions of the boundary such that $\Gamma_D \cap \Gamma_N = \partial \Omega$ and $\Gamma_D \cup \Gamma_N = \emptyset$, $\mathbf f$ is a given forcing term, $\mathbf g$ and $\mathbf h$ are the Dirichlet and Neumann boundary data respectively, and
\begin{equation}
\sigma(\mathbf u,p) = \mu \nabla \mathbf u  - pI
\end{equation}
is the stress tensor. The domain is partitioned into five subdomains $\Omega_i$ with $i = 1,\ldots,5$ and divided by four interfaces $\Gamma_i$ with $i = 1,\ldots,4$, as shown in Fig.~\eqref{fig:ns_convergence} (left). We define a family of non-conforming triangulations $\mathcal T^h$ characterized by the mesh size $h$, i.e. the maximum edge length over $\Omega$ , which corresponds to the mesh size in $\Omega_1$;  $\Omega_3$ and $\Omega_5$ are characterized by approximately the same mesh size $h$, whereas $\Omega_2$ and $\Omega_4$ have mesh size $h/2$. The choice of using meshes refined in the bottom left and bottom right subdomains is motivated by the numerical simulation presented in Section~\ref{subsec:lidcavity}, in which the solution of the Navier-Stokes equations presents recirculation zones localized in these regions of $\Omega$. Differently from what done in Section~\ref{subsec:poisson}, we only focus on non-conforming meshes and we exclusively employ the inf-sup stable Taylor-Hood \cite{hood1974navier} elements with quadratic Lagrangian basis functions for the velocity and linear Lagrangian basis functions for the pressure. The non-linear discretized system is numerically solved by Newton's method.

\begin{remark}
When applied to the Navier-Stokes equations in two dimensions, the method requires assigning to each interface two sets of basis functions discretizing the two components of the normal stress. To see why this is the case, consider the situation in which $\Omega$ is subdivided into $\Omega_1$ and $\Omega_2$; let us denote as always the interface of the two partitions $\Gamma$. Multiplying the momentum equation by a test function $\mathbf v \in [H^1_{\Gamma_D}(\Omega)]^2$ and integrating by parts on $\Omega_1$ leads to
\begin{equation}
\mu  \int_{\Omega_1} \nabla \mathbf u : \nabla \mathbf v \,\text{d}\mathbf x - \int_{\Omega_1} p \nabla \cdot \mathbf v \,\text{d}\mathbf x - \int_{\Gamma} \sigma(\mathbf u,p) \mathbf n \cdot \mathbf v \,\text{d}\mathbf{x} = \int_{\Omega_1} \mathbf f \cdot \mathbf v\,\text{d}\mathbf{x} + \int_{\partial \Omega_1 \cap \Gamma_{N}} \mathbf h \cdot \mathbf v\,\text{d} \mathbf{x}.
\end{equation}
The integral on $\Gamma$ is the coupling term. Each of the two components of the normal stress $\sigma(\mathbf u,p) \mathbf n$ must be discretized by a set of basis functions. In this paper, we choose for simplicity to use the same set for the two components of the normal stress.
\label{remark:bfns}
\end{remark}
We consider again Fourier basis functions for the approximation of the normal stresses. Since, as explained in Remark~\ref{remark:bfns}, we need two Lagrange multipliers for representing each normal stress, the number of basis functions on $\Gamma_i$ is found as $n_\Gamma^{(i)} = 2(2n_{\omega^{(i)}} + 1)$, where $n_{\omega^{(i)}}$ is the number of frequencies used on the $i^\text{th}$ interface.

\subsubsection{Numerical convergence against the exact solution}
\begin{figure}
\centering
\setlength
\figureheight{0.6\textwidth}
\setlength
\figurewidth{0.7\textwidth}
\input{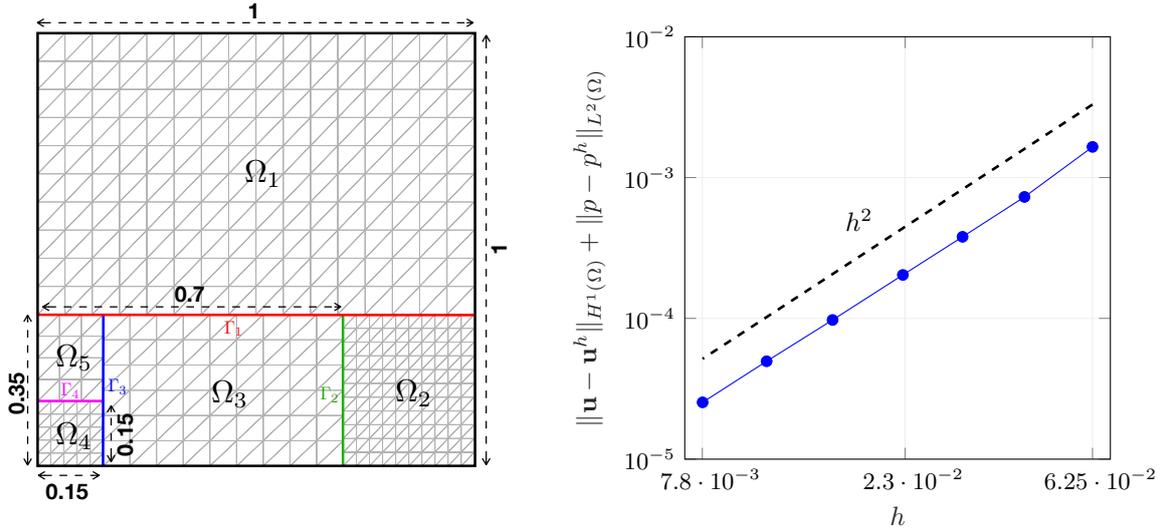}
\caption{On the left, domain decomposition and computational mesh of $\Omega = (0,1) \times (0,1)$, namely the domain considered in the numerical simulation of the Navier-Stokes equations. Each triangulation $\mathcal T^h$ is characterized by the following mesh sizes: $h$ in $\Omega_1$, $h/2$ in $\Omega_2$, $h$ in $\Omega_3$, $h/2$ in $\Omega_4$ and $h$ in $\Omega_5$.  On the right, convergence of the error $\Vert \mathbf u - \mathbf u^h \Vert_{H^1(\Omega)} + \Vert p - p^h \Vert_{L^2(\Omega)}$ against the exact solution \eqref{eq:exactns} with respect to $h$. The number of basis functions at each interface is the same for each evaluation of the error: we set $n_\Gamma^{(1)} = 22$, $n_\Gamma^{(2)} = 18$, $n_\Gamma^{(3)} = 18$ and $n_\Gamma^{(4)} = 14$.}
\label{fig:ns_convergence}
\end{figure}
\label{subsec:nsexact}
We consider Eq.~\eqref{eq:ns} with $\mu = 1$, $\Gamma_N = \{1\} \times (0,1)$, $\Gamma_D = \partial \Omega \setminus \Gamma_N$, and $\mathbf f$, $\mathbf g$, and $\mathbf h$ chosen such that
\begin{equation}
\mathbf u =
\begin{bmatrix}
\sin(y \pi) \\
\exp(x)
\end{bmatrix},
\qquad
p = - \dfrac{1}{2} x^2
\label{eq:exactns}
\end{equation}
is the exact solution.

On the interfaces we set $n_\Gamma^{(1)} = 22$, $n_\Gamma^{(2)} = 18$, $n_\Gamma^{(3)} = 18$ and $n_\Gamma^{(4)} = 14$. The number of basis functions on the interfaces is chosen such that the optimal convergence of finite elements is retrieved. Fig.~\ref{fig:ns_convergence} (right) shows that the following classical error estimate for Taylor-Hood elements
\begin{equation}
\Vert \mathbf u - \mathbf u^{h} \Vert_{H^1(\Omega)} + \Vert p - p^{h} \Vert_{L^2(\Omega)} \leq C h^2
\label{eq:convergence_ns}
\end{equation}
holds. The norms in \eqref{eq:convergence_ns} should be interpreted as broken norms.
\subsubsection{Lid-driven cavity problem}
\label{subsec:lidcavity}
\begin{figure}
\centering
\setlength
\figureheight{0.8\textwidth}
\setlength
\figurewidth{0.8\textwidth}
\input{pictures/streamlines_zoom.tikz}
\caption{Streamlines of the lid-driven cavity problem computed with  meshes belonging to the family $\mathcal T^h$, with $h = 1/16$,  $h = 1/32$ and  $h = 1/128$. The streamlines of a fine reference solution obtained on a uniform mesh with $h = 1/300$ are displayed in black dashed lines. The plots on the top right and bottom row show details over regions of $\Omega$.}
\label{fig:streamlines}
\vspace{1cm}
\begin{tabular}{l l l l l l l}
\toprule
$ $ & \multicolumn{2}{c}{dofs} & & \multicolumn{3}{c}{error} \\
\cmidrule(lr){2-3}
\cmidrule(lr){5-7}
$h$ & velocity & pressure & system size &
 E1 & E2 & E3 \\
 \midrule
1/16 & 3'214 & 445 & 3'751 & 2.9e--3 & 6.5e--4 & 8.3e--4 \\
1/32 & 12'162 & 1'602 & 13'856 & 2.4e--4 & 5.0e--5 & 8.9e--5 \\
1/64 & 46'982 & 6'031 & 53'105 & 1.5e--5 & 4.5e--6 & 8.4e--6 \\
1/128 & 184'178 & 23'334 & 207'604 & 8.3e--7 & 1.6e--6 & 1.3e--6 \\
\bottomrule
\end{tabular}
\captionof{table}{Degrees of freedom (dofs) for the meshes considered in the lid-driven cavity problem and approximation error of the center of the three eddies (E1: primary central eddy, E2: bottom right eddy, E3: bottom left eddy). The system size is the sum of the degrees of freedom of velocity and pressure and the number of degrees of freedom for the Lagrange multipliers, the latter being constant and equal to 92 for all the meshes. The center of the eddy is numerically found as the point where the minimum of the magnitude of the velocity field is reached, and the error is computed as the Euclidean distance of such approximation with the center of the eddy of the fine solution obtained over a uniform mesh with $h = 1/300$, corresponding to 722'402 degrees of freedom for the velocity and 90'601 degrees of freedom for the pressure. The coordinates of the centers of E1, E2 and E3 for the fine solution are: [0.545907,0.593810], [0.879935,0.121555] and [0.059761,0.053354] respectively.}
\label{tab:dofs_error}
\end{figure}
We now focus on the numerical approximation of the classic lid-driven cavity problem \cite{bozeman1973numerical, ghia1982high} with Reynolds number Re = 500. Specifically, we consider Eq.~\eqref{eq:ns} with $\mu = 1$, $\Gamma_D = \partial \Omega$, $\mathbf g = [U,\, 0]^T$ with $U = 500$ on $(0,1) \times \{1\}$ and $U = 0$ on the rest of the boundary, and $\mathbf f = [0,\, 0]^T$. We consider $n_\Gamma^{(1)} = 42$, $n_\Gamma^{(2)} = 18$, $n_\Gamma^{(3)} = 18$ and $n_\Gamma^{(4)} = 14$; as in the numerical simulation presented in Section~\ref{subsec:nsexact}, the number of basis functions on the interfaces is chosen such that the error on the Lagrange multipliers can be considered negligible if compared with the finite element error. We remark that, being the solution of the problem at hand considerably more difficult to capture accurately than the exact solution \eqref{eq:exactns} -- because it features steep gradients and higher Reynolds numbers -- it became necessary to increase the number of basis functions on $\Gamma_1$ in order to obtain optimal convergence. Since we consider only Dirichlet boundary conditions, the problem is not well-posed as the pressure is unique up to a constant. We deal with this issue by fixing the degree of freedom of the pressure in the bottom left corner to zero.

Fig.~\ref{fig:streamlines} shows the streamlines obtained by solving the problem with non-conforming meshes belonging to the family $\mathcal T^h$ characterized by $h = 1/16$, $h = 1/32$ and $h = 1/128$. Furthermore, the streamlines of a fine solution computed with uniform $h = 1/300$ are displayed for reference: these are qualitatively similar to the ones corresponding to Re = 500 reported in e.g. \cite{shi2002laminar}, and we, therefore, assume that the fine solution well approximates the exact solution of the problem. As shown in Fig.~\ref{fig:streamlines} (bottom right), using a smaller mesh size in the regions of the two smaller eddies in the lower part of the domain allowed us to obtain satisfactory approximations of those secondary recirculation zones even with the coarsest mesh size $h = 1/16$. As expected, the differences in the streamlines among the different refinement levels are more evident in $\Omega_1$, were for each mesh belonging to $\mathcal T^h$ the largest elements are located. Specifically, we notice that, while the position of the primary eddy is approximated with good accuracy even for $h = 1/16$, in the peripheral regions of the domain only the streamlines corresponding to $h = 1/128$ are almost indistinguishable from the ones of the exact solution; see Fig.~\ref{fig:streamlines} (top right)
and Fig.~\ref{fig:streamlines} (bottom left). Table~\ref{tab:dofs_error} provides the number of finite element degrees of freedom corresponding to each refinement level and the error in the approximation of the center of each eddy: this quantity is computed as the Euclidean distance of the points where the velocity field attains minimum velocity (in magnitude) in the coarse solutions and in the fine solution. The errors become smaller with $h$. We remark that, for coarse meshes, the approximation error of the two smaller eddies in the lower part of $\Omega$ is one order of magnitude lower than that of the central eddy; this confirms that employing the non-conforming meshes in $\mathcal T^h$, which are characterized by a smaller element size in $\Omega_2$ and $\Omega_4$, leads to a satisfactory approximation of the secondary recirculation regions even for large values of $h$. It is worth noting that the number of degrees of freedom reserved to the discretization of the Lagrange multipliers is constant for all the meshes (as it depends solely on the number of basis functions at each interface) and it is equal to 92: this quantity is much smaller than the number of degrees of freedom for velocity and pressure for each $h$. Hence, the coupling of the finite element spaces is performed by introducing a negligible number of additional variables.

%

\section{Conclusions}
We presented a non-conforming domain decomposition method for non-overlapping subdomains. At the continuous level, our method and the mortar method are based on the same weak formulations in which the continuity constraints over the primal (the solution) and the dual (the stresses) variables are enforced via Lagrange multipliers. As we described in the paper, our choice of discretizing the space of Lagrange multipliers independently of the spatial discretization in the subdomains offers the advantage of a straight-forward implementation of the method and the possibility of tuning the accuracy of the coupling as required by the application; we limited ourselves to considering Fourier basis functions defined over the interface. However, the saddle-point nature of the problem poses constraints over the richness of the discretized space for the Lagrange multipliers compared to the degrees of freedom of the primal variable: we empirically verified that the inf-sup constant can be controlled dependently on the mesh size, in the sense that finer meshes allow considering larger number of Fourier basis functions, without violating the inf-sup stability. In the numerical experiments, we showed that the optimal convergence of the finite element method was recovered for the Poisson problem; this was confirmed both when using conforming and non-conforming meshes, and when using different polynomial degrees in the subdomains. In the last part of the paper, we showed that the method can be easily extended to the case of non-elliptic equations, such as the Navier-Stokes equations, and to cases of partitions of the domain into multiple subdomains. We were able to recover the optimal convergence rate of finite elements also for the Navier-Stokes equations by considering a number of basis functions on the interfaces considerably lower than the number of degrees of freedom of the discretized subdomains. Moreover, we focused on a possible practical application of the method, i.e. the use of non-conforming structured meshes in the lid-driven cavity problem for capturing the secondary recirculation regions. We showed that considering smaller mesh sizes in correspondence of the secondary eddies ensures satisfactory results in terms of accuracy of the streamlines of the vortex rings.
\label{sec:conclusions}

\section*{Acknowledgments}
The authors are grateful to Prof. Annalisa Buffa and Prof. Alfio Quarteroni for the fruitful discussions and their advice on the topics presented in this paper. The research of the authors is supported by the Swiss National Foundation (SNF), project No. 140184.

\bibliographystyle{model1-num-names}
\bibliography{references}

\end{document}